\documentclass[11pt]{article}
\usepackage{a4wide}
\usepackage[utf8]{inputenc}
\usepackage{multirow}
\usepackage{tikz}
\usetikzlibrary{positioning}
\usepackage[a4paper, margin=2.3cm]{geometry}
\usepackage{amsmath,amssymb,amsthm}
\usepackage{thmtools}
\usepackage{thm-restate}
\usepackage{xcolor}
\usepackage[normalem]{ulem}
\usepackage{parskip}
\usepackage{hyperref}
\usepackage{setspace}
\usepackage{makecell}
\usepackage{graphicx}
\usepackage{enumitem}
\usepackage{comment}
\usepackage{float}

\usepackage{yhmath}
  
\newtheorem{theorem}{Theorem}[section]
\newtheorem{remark}[theorem]{Remark}
\newtheorem{corollary}[theorem]{Corollary}

\newtheorem{lemma}[theorem]{Lemma}
\newtheorem{question}[theorem]{Question}
\newtheorem{proposition}[theorem]{Proposition}

\newtheorem*{definition*}{Definition}
\newtheorem*{remark*}{Remark}

\hypersetup{
  pdftitle={Subset-Sum Density Realization in Locally Finite Abelian Groups},
  pdfauthor={Norbert Hegyv\'ari, Thang Pham, Boqing Xue}
}

\newcommand{\PP}{\mathcal P}
\newcommand{\F}{\mathbb F}
\newcommand{\N}{\mathbb N}
\newcommand{\Z}{\mathbb Z}
\newcommand{\Span}{\operatorname{span}}

\newcommand{\Deleted}[1]{{\color{red}\begingroup
  \let\DeletedCite\cite
  \def\cite##1{\mbox{\DeletedCite{##1}}}%
  \let\DeletedRef\ref
  \def\ref##1{\mbox{\DeletedRef{##1}}}%
  \sout{#1}\endgroup}}

\title{Subset-Sum Density Realization in Locally Finite Abelian Groups}
\author{Norbert Hegyv\'ari\thanks{E\"otv\"os University and associated member of Alfr\'{e}d R\'{e}nyi Institute, Hungary. Email: hegyvari@renyi.hu} \and Thang Pham \thanks{Institute of Mathematics and Interdisciplinary Sciences at Xidian University, China. \newline
\hspace*{0.5cm} Email: phamanhthang.vnu@gmail.com} \and Boqing Xue\thanks{Institute of Mathematical Sciences, ShanghaiTech University. Email: xuebq@shanghaitech.edu.cn}}
\date{}

\begin{document}

\maketitle

\begin{abstract}
Let $G$ be a countable locally finite abelian group, and let
\[
G_1\leq G_2\leq\cdots,
\qquad
\bigcup_{i\geq1}G_i=G,
\]
be any filtration of $G$ by finite subgroups. For $A\subseteq G$, let $\PP(A)$ denote the set of all finite subset sums of elements of $A$, and let $2G:=\{2g:g\in G\}$. We provide the sharp classification of the groups in which every closed interval $[\alpha,\beta]\subseteq[0,1]$ can be realized, along any fixed filtration by finite subgroups, as the exact set of subsequential limits of
\[
 \left(\frac{|\mathcal{P}(A)\cap G_i|}{|G_i|}\right)_{i\geq1}
\]
for a suitable set $A\subseteq G$. This realization property holds if and only if $|2G|=\infty$. We also prove that, whenever $G$ is infinite, every subset-sum set $\mathcal{P}(A)$ of positive upper density along the given filtration contains an infinite set $B$ with $B+B\subseteq\mathcal{P}(A)$. Consequently, the interval realizations above also have this additional property whenever $\beta>0$. 
\end{abstract}
\tableofcontents
\section{Introduction}\label{sec:introduction}
Let $G$ be an abelian group written additively. For a subset $A$ of $G$, write
\[
\PP(A) := \left\{\sum_{a\in F}a: \, F\subseteq A,\, |F|<\infty\right\}
\]
for its set of finite subset sums. We include the empty sum, so that $0\in\PP(A)$. In questions about asymptotic density in the positive integers, this convention changes the set by at most one point and is therefore immaterial.

For $S\subseteq\mathbb N$, its lower and upper asymptotic densities are
\[
d_n(S):=\frac{|S\cap[1,n]|}{n},\qquad
\underline d(S) :=\liminf_{n\to\infty}d_n(S),
\qquad
\overline d(S) :=\limsup_{n\to\infty}
d_n(S).
\]

There is a simple continuity feature of density on $\mathbb N$, namely, 
\[
 |d_{n+1}(S)-d_n(S)|\leq\frac1{n+1}.
\]
Therefore, the set of subsequential limits of $(d_n(S))_{n\geq1}$ is the entire interval $[\underline d(S),\overline d(S)]$. Thus, on the integers, the set of limit points is completely determined by its two endpoints.

Answering a question posed by Zannier, Ruzsa \cite{Ruzsa91}
proved that if $A=\{a_1<a_2<\cdots\}\subseteq\mathbb N$
satisfies $a_{n+1}\leq2a_n$ for all but finitely many $n$,
then $\PP(A)$ has an asymptotic density. Ruzsa also asked
whether, without this growth assumption, the lower and upper
densities of $\PP(A)$ could take arbitrary admissible values.

\begin{question}[Ruzsa]\label{q:ruzsa}
For every pair of real numbers $0\leq\alpha\leq\beta\leq1$, does there exist $A\subseteq\mathbb N$ such that $\underline d\bigl(\PP(A)\bigr)=\alpha$ and $\overline d\bigl(\PP(A)\bigr)=\beta$?
\end{question}

In \cite{Heg95}, the first author answered this question affirmatively:

\begin{theorem}[Hegyv\'ari, 1995]\label{thm:hegyvari}
For every pair $0\leq\alpha\leq\beta\leq1$, there exists $A\subseteq\mathbb N$ such that 
\[
\underline d\bigl(\PP(A)\bigr)=\alpha, \qquad \overline d\bigl(\PP(A)\bigr)=\beta.
\]
\end{theorem}

Theorem~\ref{thm:hegyvari} is the starting point of this research. We ask what remains of this realization phenomenon when the initial intervals $[1,n]$ are replaced by an arbitrary filtration of a countable locally finite abelian group by finite subgroups.

\subsection{Filtered locally finite groups and the main results}

An abelian group \(G\) is \emph{locally finite} if every finitely generated subgroup is finite.  Let \(G\) be countable and locally finite. By a \emph{filtration} of \(G\) we mean a chain
\begin{equation}\label{eq:def-filtration}
 G_1\leq G_2\leq\cdots,  \qquad  \bigcup_{i\geq1}G_i=G,
\end{equation}
of finite subgroups. 

Unlike the initial intervals in $\mathbb N$, a countable locally finite group may admit many different filtrations by finite subgroups, and the resulting density may depend on the choice. We therefore fix $G_\bullet=(G_i)_{i\geq1}$ in advance and regard the pair $(G,G_\bullet)$ as the object under consideration. Each index $i$ is called a \emph{level} of the filtration.

For $S\subseteq G$ we define
\[
d_i^{G_\bullet}(S):=\frac{|S\cap G_i|}{|G_i|}, \qquad \underline d_{G_\bullet}(S):=\liminf_{i\to\infty}d_i^{G_\bullet}(S),  \qquad \overline d_{G_\bullet}(S):=\limsup_{i\to\infty}d_i^{G_\bullet}(S).
\]

The passage from $\mathbb N$ to an arbitrary  filtration changes the realization problem in an essential way. The indices $[G_{i+1}:G_i]$ may be arbitrarily large, so there is no analogue of the estimate $|d_{n+1}(S)-d_n(S)|\leq 1/(n+1)$. Even if the lower and upper limits of a density sequence are known, its intermediate values need not occur as subsequential limits.

We therefore keep track of the entire set of subsequential limits of the subset-sum densities:
\begin{equation}\label{eq:def-limit-set}
\mathcal L_{G_\bullet}(A):= \left\{\gamma\in[0,1]:\, 
 \lim\limits_{\ell\rightarrow \infty} d_{i_\ell}^{G_\bullet}(\PP(A))= \gamma
 \text{ for some }i_1<i_2<\cdots\right\}.
\end{equation}

For a fixed filtration $G_\bullet$, we ask whether every interval $[\alpha,\beta]\subseteq[0,1]$ can occur as $\mathcal L_{G_\bullet}(A)$ for some $A\subseteq G$. The structural question is to determine which algebraic feature of $G$ makes this possible for every filtration.

Our answer is a sharp dichotomy governed by the set of doubles
\[
 2G:=\{2g:g\in G\}.
\]
Although the individual densities depend on the chosen filtration, the possibility of realizing every interval along every filtration is determined solely by whether $2G$ is infinite. We first state the flexible half of the dichotomy.

\begin{theorem}\label{thm:main}
Let \(G\) be a countable locally finite abelian group with \(|2G|=\infty\). Let $G_\bullet=(G_i)_{i\geq 1}$ be an arbitrary given filtration by finite subgroups:
\[
 G_1\leq G_2\leq\cdots,
 \qquad
 \bigcup_{i\geq1}G_i=G.
\]
Then, for every
\(0\leq\alpha\leq\beta\leq1\), there exists \(A\subseteq G\) such that
\[
 \mathcal L_{G_\bullet}(A)=[\alpha,\beta].
\]
\end{theorem}

In Theorem~\ref{thm:main}, the filtration is fixed in advance and is subject to no regularity assumption. The set \(A\) is constructed afterwards and may depend on this filtration. Moreover, the conclusion specifies the entire cluster set of the density sequence, rather than only its lower and upper limits.

The opposite behavior is already visible in an elementary abelian $2$-group. If $2G=\{0\}$, then $G$ is an $\F_2$-vector space and
\[
 \PP(A)=\Span_{\F_2}(A).
\]
Hence $\PP(A)\cap G_i$ is a subspace of $G_i$, and its relative size is $2^{-c_i}$, where $c_i$ is its codimension. These codimensions are nondecreasing, so the density sequence converges to either $0$ or $2^{-m}$ for some nonnegative integer $m$. The following theorem shows that a quantitative form of this rigidity persists whenever $2G$ is finite.

\begin{theorem}\label{thm:finite-2G-obstruction}
Let \(G\) be a countable locally finite abelian group for which \(|2G|<\infty\), and let \(G_\bullet\) be any filtration by finite subgroups.
For every \(A\subseteq G\), there exists \(m\in\Z_{\geq 0}\cup\{\infty\}\) such that
\[
 \frac{2^{-m}}{|2G|}
 \leq \underline d_{G_\bullet}\bigl(\PP(A)\bigr)
 \leq \overline d_{G_\bullet}\bigl(\PP(A)\bigr)
 \leq 2^{-m},
\]
where \(2^{-\infty}:=0\).  In particular,
\[
 \overline d_{G_\bullet}\bigl(\PP(A)\bigr)>\frac12
 \quad\Longrightarrow\quad
 \underline d_{G_\bullet}\bigl(\PP(A)\bigr)
 \geq\frac1{|2G|}.
\]
\end{theorem}

The conclusion is stronger than the exclusion of a single pair of densities. If $2G$ is finite and the lower density of $\PP(A)$ is zero, then Theorem~\ref{thm:finite-2G-obstruction} forces its upper density to be zero as well. Thus no pair $(0,\beta)$ with $\beta>0$ can occur. When $2G=\{0\}$, the two bounds in the theorem coincide, recovering convergence to $0$ or to a reciprocal power of two. Combining the flexible and rigid cases gives the following exact characterization.

\begin{corollary}\label{cor:sharp-criterion}
For a countable locally finite abelian group \(G\), the following statements are equivalent.
\begin{enumerate}
\item[(\romannumeral1)] \(2G\) is infinite.

\item[(\romannumeral2)] For every filtration \(G_\bullet\) and every \(0\leq\alpha\leq\beta\leq1\), there exists \(A\subseteq G\)
such that
\[
 \underline d_{G_\bullet}\bigl(\PP(A)\bigr)=\alpha,
 \qquad
 \overline d_{G_\bullet}\bigl(\PP(A)\bigr)=\beta.
\]

\item[(\romannumeral3)] For every filtration \(G_\bullet\) and every \(0\leq\alpha\leq\beta\leq1\), there exists \(A\subseteq G\)
such that
\[
 \mathcal L_{G_\bullet}(A)=[\alpha,\beta].
\]
\end{enumerate}
\end{corollary}

Thus, our realization result is not merely an analogue of the Hegyv\'ari construction, but a sharp classification theorem: it identifies \(|2G|=\infty\) as exactly the algebraic condition under which, for every given filtration $G_\bullet$, every interval in \([0,1]\) occurs as the full cluster set \(\mathcal L_{G_\bullet}(A)\) for some \(A\subseteq G\).

The positive-upper-density regime of Theorem~\ref{thm:main} carries an additional structural consequence.  For \(B\subseteq G\), write $B+B:=\{b+b':b,b'\in B\}$. We prove in Section~\ref{sec:positive-density-BB} that every finite-subset-sum set \(\PP(A)\) of positive upper density along \(G_\bullet\) contains \(B+B\) for some infinite set \(B\subseteq\PP(A)\).  Applying this principle to the realizing sets supplied by Theorem~\ref{thm:main} gives the following strengthening.

\begin{theorem}
\label{thm:realization-BB}
Let $G$ be a countable locally finite abelian group with $|2G|=\infty$, and let $G_\bullet$ be a fixed filtration.  For every $0\leq\alpha\leq\beta\leq1$ with $\beta>0$, there exist $A\subseteq G$ and an infinite set $B\subseteq\PP(A)$ such that
\[
 \mathcal L_{G_\bullet}(A)=[\alpha,\beta]
 \qquad\text{and}\qquad
 B+B\subseteq\PP(A).
\]
\end{theorem}

Taking $\alpha=0$ in Theorem~\ref{thm:realization-BB}, we obtain that, for every $0<\beta\leq1$, there exist $A\subseteq G$ and an infinite set
$B\subseteq\PP(A)$ such that $\mathcal L_{G_\bullet}(A)=[0,\beta]$ and $B+B\subseteq\PP(A)$. Thus $\PP(A)$ may have lower density zero while containing $B+B$ for some infinite set $B$.

\subsection{Relation to previous work}

Three lines of previous work are directly relevant. The first concerns the same type of realization problem. For subset sums on $\mathbb N$, this is precisely the Ruzsa--Hegyv\'ari problem \cite{Heg95, Ruzsa91} recalled above. For ordinary sumsets, Bienvenu and Hennecart \cite{BienHen20} solved the analogous problem of simultaneously realizing the asymptotic densities of $A$ and $A+A$.

The second line concerns the same class of groups and the same way of measuring density. Bienvenu and Hennecart \cite{BienHen22} considered an abelian group $G=\bigcup_{i\geq1}G_i$ filtered by finite subgroups and defined density through the ratios $|S\cap G_i|/|G_i|$, exactly as above. In this setting they proved a Kneser-type theorem for ordinary sumsets. Theorem~\ref{thm:main} brings the first two lines together: it places the Ruzsa--Hegyv\'ari realization problem for $\PP(A)$ in countable locally finite abelian groups, with the filtration fixed in advance, and proves that the possibility of realizing every interval in $[0,1]$ as $\mathcal L_{G_\bullet}(A)$ for every such filtration is equivalent to $|2G|=\infty$.

A third line of work concerns infinite additive configurations inside sets of positive density.  Moreira, Richter and Robertson \cite[Theorem~1.2]{MRR19Sumset} proved that every subset of $\mathbb N$ of positive upper Banach density contains $B+C$ for two infinite sets $B,C$. Extending this
result, Kra, Moreira, Richter and Robertson \cite[Theorem~1.1]{KMRR24Infinite} proved that, for every fixed $k$, every subset of $\mathbb N$ of positive upper Banach density contains
\[
 B_1+\cdots+B_k
\]
for some infinite sets $B_1,\ldots,B_k\subseteq\mathbb N$. The same authors \cite[Theorem~1.2]{KMRR24BBt} subsequently proved that every subset of $\mathbb N$ of positive upper Banach density contains a translate of
\[
 \{b+b':b,b'\in B,\ b\neq b'\}
\]
for some infinite set $B$. Their density finite sums theorem extends this conclusion: for each fixed $k$, one obtains a
translate of all sums of between one and $k$ distinct elements of an infinite set \cite[Theorem~1.1]{KMRR26DensityFS}. Charamaras and Mountakis
obtained a group extension of the shifted restricted-sumset result. They \cite[Corollary~1.13]{CM25} proved that, if $G$ is a countably infinite abelian group satisfying $[G:2G]<\infty$, then every set $E\subseteq G$ of positive upper Banach density contains a translate of the restricted sumset of some infinite set $B\subseteq E$. For infinite $G$, the condition $[G:2G]<\infty$ implies $|2G|=\infty$, but the converse need not hold. Their result concerns a shifted restricted sumset in an arbitrary positive-density set and is therefore complementary to Theorem~\ref{thm:positive-density-BB}, which obtains an unshifted unrestricted sumset from the special structure of $\PP(A)$ without any hypothesis on $2G$. For unrestricted sumsets of the form $B+B$, which include the diagonal sums $2b$, nontrivial density thresholds arise already in the integers (see \cite{KR25BB}), and recent work studies sharp thresholds in countable abelian groups under suitable hypotheses on the doubling map and the F{\o}lner sequence (see \cite{CKMR25,Kousek26}).  Most relevant to our setting, Charamaras, Kousek, Mountakis and Radi\'c \cite[Corollary~5.6]{CKMR25} show that, for every countable abelian group $G$ with $2G$ infinite, there exists a set $E\subseteq G$ of upper Banach density one such that
\[
 t+B+B\nsubseteq E
\]
for every $t\in G$ and every infinite $B\subseteq G$. Thus upper Banach density one alone does not force an unrestricted $B+B$ configuration. Theorem~\ref{thm:positive-density-BB}, by contrast, exploits the special structure of $\PP(A)$ as a set of finite subset sums. For a broader account of infinite-sumset problems, their obstructions, and their extensions beyond the integers, see the survey \cite[Sections~2 and~5]{KMRR25Survey}.

The hypothesis $|2G|=\infty$ is a common structural condition behind several apparently different additive phenomena. Fern\'andez-Bret\'on, Sarmiento Rosales, and Vera proved that, for an abelian group $G$, it is equivalent to the following finite Ramsey property: for every pair of positive integers $n,r$, every coloring $c:G\to\{1,\ldots,r\}$ admits a set $X\subseteq G$ with $|X|=n$ such that the full sumset $X+X=\{x+y:x,y\in X\}$ is monochromatic \cite[Theorem~1]{FSRV24}. For locally finite groups, the same condition is also equivalent to the assertion that every finite coloring of $G$ contains a proper monochromatic three-term arithmetic progression, that is, pairwise distinct $x,y,z\in G$ satisfying $x+z=2y$. Indeed, if $2G$ is finite, the finite coloring $g\mapsto 2g$ rules out such a progression. In the opposite direction, the assertion follows from van der Waerden's theorem \cite{vdW27} together with Lev's bound for true-progression-free subsets of finite abelian groups \cite[Theorem~1]{Lev04}. These statements are not direct consequences of one another; rather, they are different manifestations of the same structural condition. Theorem~\ref{thm:main} shows that complete realization of subset-sum densities along  each fixed filtration is governed by this condition as well.

\subsection{Outline of the proofs}

The proof of Theorem~\ref{thm:main} begins with an exact transfer principle.  If \(\pi:G\twoheadrightarrow Q\) is surjective and \(Q_i:=\pi(G_i)\), then every \(C\subseteq Q\) admits a lift \(A\subseteq G\) satisfying $\PP(A)=\pi^{-1}\bigl(\PP(C)\bigr)$, and hence
\[
 \frac{|\PP(A)\cap G_i|}{|G_i|}
 =
 \frac{|\PP(C)\cap Q_i|}{|Q_i|}
 \qquad(i\geq1).
\]
The equality is termwise, so the whole cluster set is preserved. The lift is concrete: one takes the entire kernel of \(\pi\), together with one representative of each nonzero element of \(C\).

The structural reduction then shows that every countable locally finite abelian group with \(|2G|=\infty\) has a quotient isomorphic to at least one of
\[
 \bigoplus_{n\geq1}(\Z/4\Z),
 \qquad
 \bigoplus_{n\geq1}(\Z/p\Z),
 \qquad
 \bigoplus_{n\geq1}(\Z/p_n\Z),
 \qquad
 \Z(\tilde{p}^\infty),
\]
where \(p\) denotes some odd prime, the \(p_n\) are pairwise distinct odd primes, and $\tilde{p}$ is some prime, with $\Z(\tilde{p}^\infty)$ being the Pr\"{u}fer \(\tilde{p}\)-group. These are quotient alternatives rather than a classification of \(G\), and they need not be mutually exclusive. They divide naturally into two constructive regimes: the first two have an infinite coordinate structure, whereas the last two are locally cyclic.

We first indicate the locally cyclic construction. Here every finite filtration subgroup is cyclic.  Suppose \(G_i<G_j\), let \(q=[G_j:G_i]\), and let \(x\) generate \(G_j\). For each integer $11\leq m<q$, a finite integer subset-sum construction produces a set \(E\subseteq G_j\setminus G_i\) satisfying
\[
 \PP(C\cup E)
 =
 \PP(C)+\{0,x,2x,\ldots,(m-1)x\}
\]
for any $C\subseteq G_i$. If the density of \(\PP(C)\) in \(G_i\) is \(u\), then the density of $\PP(C\cup E)$ in \(G_j\) is exactly $um/q$. At every intermediate filtration level the density remains between \(um/q\) and \(u\). Taking \(q\) large and choosing \(m\) appropriately therefore lowers the density to within \(O(q^{-1})\) of any given smaller value.

The reverse operation uses a second finite subset-sum construction. In a sufficiently large cyclic extension, it produces either one translate or two adjacent translates of the preceding set in each new coset, with two translates in all but a bounded number of cosets.  The sets under construction have interval sections in the old cosets, and this property is preserved.  A quantitative lower bound shows that each extension gains a fixed positive proportion of the remaining gap to density \(1\). Iteration therefore crosses every target below \(1\).  A complementary upper bound makes the excess above the target at the first crossing arbitrarily small.

The locally cyclic argument alternates these two operations.  In cycle \(k\), an interval contraction brings the density within \(O(q_k^{-1})\) of \(\alpha\), and finitely many two-translate extensions stop at the first crossing of a target tending to \(\beta\), where \(q_k\to\infty\). The same first-crossing argument applied to any \(\gamma\in(\alpha,\beta)\) produces a density within \(O(q_k^{-1})\) of \(\gamma\).  Each extension set is separated from the preceding subgroup, so later cycles create no new subset sums at levels already treated.

The coordinate quotients require a different construction. We first choose a basis adapted simultaneously to the given filtration. Over \(\F_p\), this is an adapted vector-space basis. Over \(\Z/4\Z\), the chosen basis elements must additionally record that the order-two subgroup of a cyclic coordinate may enter the filtration before the full coordinate.  On a finite coordinate summand \(U\cong \mathcal R^m\), where \(\mathcal R=\F_p\) or \(\mathcal R=\Z/4\Z\), we construct a finite set \(B\subseteq U\) and an element \(v\in U\) such that
\[
 \PP(B)=U\setminus\{v\}.
\]
Thus the local density factor is \(1-|U|^{-1}\).  The support is placed so that \(v\in G_i\) exactly when \(U\subseteq G_i\), which makes this factor compatible with every level of the fixed filtration. Auxiliary coordinates attached to the same missing element reduce its missing proportion geometrically.  At a level where \(r\) auxiliary coordinates contribute, the factor becomes
\[
 1-p^{-(m+r)}
 \quad\text{if }\mathcal R=\F_p,
 \qquad
 1-4^{-m}2^{-r}
 \quad\text{if }\mathcal R=\Z/4\Z.
\]
Since distinct constructions use disjoint direct summands, then their local density factors multiply exactly.  At each fixed filtration level only finitely many factors occur, and basis elements placed after that level cannot alter its density.

The product formula allows us to alternate these two operations. In cycle $k$, co-singleton blocks lower the density close to $\alpha$, and auxiliary coordinates then raise it to the first crossing of a target tending to $\beta$. At each increasing step, we choose a block with largest current missing proportion, which makes the excess above the target tend to zero. The small changes during the decreasing steps realize every value in $[\alpha,\beta]$, while stability ensures that later cycles do not alter densities at earlier filtration levels.

The additional \(B+B\) conclusion in Theorem~\ref{thm:realization-BB} is supplied by the general positive-density principle stated as Theorem~\ref{thm:positive-density-BB}.  To prove this principle, we first observe that positive upper density allows us to reduce to the case \(\langle A\rangle=G\). Let \(S:=\PP(A)\), and let \(R\) be the subgroup of all elements that admit a finite-subset-sum representation using \(A\setminus D\) for every finite \(D\subseteq A\).  If \(R\) is infinite, then \(R\subseteq S\) and \(S+R=S\), so \(B:=R\) has the required property.  If \(R\) is finite, we quotient by \(R\) and encode simultaneous representations of \(x\) and \(2x\) by a countable multihypergraph on \(A\).  An infinite matching yields infinitely many support-disjoint certificates and hence an infinite set \(B\subseteq S\) with \(B+B\subseteq S\).  If no infinite matching exists, a finite transversal \(U\subseteq A\) leaves a quotient subset-sum set \(T\), generated by \(A\setminus U\), such that
$$
 x,2x\in T\quad\Longrightarrow\quad x=0.
$$
The density-decay lemma forces \(T\) to have density zero.  Since the image of \(S\) is the finite union
$$
 \pi(S)=\pi\bigl(\PP(U)\bigr)+T
$$
of translates of \(T\), this contradicts the assumed positive upper density. The density-decay lemma is proved by decomposing the indexed quotient family into disjoint minimal zero-sum circuits and residual indices, and then deriving density bounds by counting suitable pairwise disjoint sets inside the finite subgroups.

\subsection{Organization of the paper}

The next section establishes the finite integer subset-sum constructions used in both later arguments. Section~\ref{sec:structural-reduction} proves the exact quotient-transfer principle and the four-type quotient alternatives.  Section~\ref{sec:locally-cyclic} treats the two locally cyclic quotients through interval contractions and two-translate extensions.  Section~\ref{sec:coordinate-cases} treats the coordinate quotients through co-singleton constructions, auxiliary coordinates, and an exact product formula. Section~\ref{sec:proofs} lifts these constructions to the original group, proves the finite-\(2G\) obstruction, and deduces Corollary~\ref{cor:sharp-criterion}. Section~\ref{sec:positive-density-BB} proves Theorem~\ref{thm:positive-density-BB} and Theorem~\ref{thm:realization-BB}. Sections 8 and 9 give concrete examples of locally abelian groups and pose a question on prescribed quadratic distances, respectively.

\section{Basic lemmas for subset sums}

The set of natural numbers is denoted by $\N=\{1,2,3,\ldots\}$. The cardinality of a finite set $E$ is denoted by $|E|$ or $\#E$.  

\begin{lemma} \label{lem_basic1}
Let \(0<a_1<\cdots<a_s\) be integers satisfying
\[
a_1=1, \qquad a_j\leq1+\sum_{h<j}a_h\quad(2\leq j\leq s).
\]
Then 
\[
\PP\big(\{a_1,a_2,\ldots,a_s\}\big) = 
\{0,1,\ldots,a_1+\cdots+a_s\}.
\]
\end{lemma}

\begin{proof}
This follows immediately by induction, because the set $\{0,1,\ldots,a_1+\ldots+a_{j-1}\}$ and its translate by \(a_j\) $(2\leq j\leq s)$ overlap or are adjacent. This completes the proof.
\end{proof}

\begin{lemma} \label{lem_subsetsum_complete_interval}
For any $t\in \N$ with $t\geq 10$, there is a subset $D\subseteq \N$ such that $\PP(D)=\{0,1,\ldots,t\}$. 
\end{lemma}

\begin{proof}
For $t=14$, take $D=\{1,2,4,7\}$. In all other cases, let \(k\) be the unique natural number such that $t=k(k+1)/2+r$ for some $0\leq r\leq k$. Let
\[
D=\{1,2,\ldots,k-1,k+r\}. 
\]
To show $\PP(D)=\{0,1,\ldots, t\}$ by Lemma \ref{lem_basic1}, it is sufficient to verify $k+r\leq1+k(k-1)/2$. One checks that it holds for \(k\geq 5\), and also for \(k=4, r\leq 3\). Thus, the conclusion holds when $t\geq 15$ or $10\leq t\leq 13$. This completes the proof. 
\end{proof}

\begin{corollary} \label{cor_missing_one}
Let $p$ be an odd prime. Then there exists a set $D\subseteq \F_p^*$ and an element $a\in \F_p^*$ such that $\PP(D)=\F_p\setminus\{a\}$.
\end{corollary}

\begin{proof}
For \(p=3,5,7,11\), take $D$ to be 
\[
\{1\},\quad
 \{1,2\},\quad
 \{2,3,4\},\quad
 \{1,3,5,6\},
\]
with the missing element \(2,4,1,2\),  respectively. 
For \(p\geq 13\), applying Lemma \ref{lem_subsetsum_complete_interval} with $t=p-2$ gives a set $D\subseteq \{1,2,\ldots,p-2\}$ with $\PP(D)=\{0,1,\ldots, p-2\}$. The missing element is $p-1$. This proves the corollary.
\end{proof}

\begin{lemma} \label{lem_missing_two}
Let $T\geq 32$ be an integer. There is a set $W\subseteq \N$ with $|W|\geq 2$ such that
\[
\PP(W)=\{0,1,\ldots,T\}\setminus \{1,T-1\}.
\]
\end{lemma}

\begin{proof}
Let $\tau_j:=2+3+\ldots+j=(j-1)(j+2)/2$ for $j\in \N$. Let $k$ be the unique natural number satisfying $T=\tau_k+r$, with $0\leq r\leq k$. Since $T \geq 32$ and $\tau_7=27$, one has $k\geq 7$. We take $W=\{2,3,\ldots,k-1,k+r\}$.

We use the following elementary observation. If a set \(B\subseteq \N\) satisfies $\PP(B)=\{0,1,\ldots,V\}\setminus \{1,V-1\}$ for some $V\geq 9$, then
\[
\PP(B\cup\{a\}) =\{0,1,\ldots,V+a\}\setminus \{1,V+a-1\} 
\]
whenever \(a\notin B\) and \(2\leq a\leq V-3\). Indeed,
\[
\PP(B\cup\{a\})
=\PP(B)\cup\bigl(a+\PP(B)\bigr),
\]
and the assumption \(a\leq V-3\) ensures that the missing element $V-1$ in $\PP(B)$ occurs in $a+\PP(B)$.

Consider the situation $V=\tau_{j-1}$ and $B=\{2,3,\ldots,j-1\}$ for $j\geq 5$. When $j=5$, one checks that $\PP(\{2,3,4\})
=\{0,1,\ldots,9\}\setminus \{1,8\}$. Since $j\leq \tau_{j-1}-3$ when $j\geq 5$, the observation may be applied successively to \(j=5,6,\ldots,k-1\) taking $a=j$. Thus, 
\[
\PP\big(\{2,3,\ldots, k-1\}\big)=\{0,1,\ldots,\tau_{k-1}\}\setminus\{1,\tau_{k-1}-1\}.
\]

Finally, since \(r\leq k\) and \(k\geq7\), one has $k+r\leq2k\leq\tau_{k-1}-3$. Applying the observation once more with $V=\tau_{k-1}$, $B=\{2,3,\ldots,k-1\}$ and $a=k+r$, and noting that $\tau_{k-1}+a=\tau_k+r=T$, gives
\[
\PP(W)=\PP\big(\{2,3,\ldots, k-1\}\cup \{k+r\}\big)=\{0,\ldots, T\}\setminus \{1,T-1\}.
\]
This completes the proof. 
\end{proof}

\begin{lemma} \label{lem_basic2}
Let $p$ be an odd prime. Let $1\leq s\leq p-1$. Then 
\[
\left\{\sum\limits_{t\in S} t:\, S\subseteq \F_p,\, |S|=s\right\} =\F_p.
\]
\end{lemma}

\begin{proof}
Fix any set $S_0\subseteq \F_p$ of size $s$, and denote $u_0:=\sum_{t\in S_0}t$. Now for any $u\in \F_p$, let $c_u$ be the element in $\F_p$ such that $sc_u=u-u_0$. Then the set $S_u=S_0+c_u$ also has size $s$, and 
\[
\sum_{x\in S_u}x = \sum_{t\in S_0}(t+c_u)
= u_0+sc_u =u.
\]
This completes the proof. 
\end{proof}

\section{Quotient reduction and structural alternatives}\label{sec:structural-reduction}

In this section, we transfer the realization problem to suitable quotients without changing the density sequence.

\begin{lemma}\label{lem:quotient-transfer}
Let $\pi:G\twoheadrightarrow Q$ be a surjective homomorphism. Give $Q$ the image filtration $Q_\bullet=(Q_i)$ by $Q_i=\pi(G_i)$ $(i\geq 1)$. For
every $C\subseteq Q$, there is a set $A\subseteq G$ such that $\PP(A)=\pi^{-1}\bigl(\PP(C)\bigr)$ and
\[
\frac{|\PP(A)\cap G_i|}{|G_i|} = \frac{|\PP(C)\cap Q_i|}{|Q_i|},\qquad (i\geq 1).
\]
Consequently, $\mathcal L_{G_\bullet}(A)=\mathcal L_{Q_\bullet}(C)$.
\end{lemma}

\begin{proof}
Write $N:=\ker\pi$. For every nonzero $c\in C$, choose one lift $\widetilde c\in G$ satisfying $\pi(\widetilde c)=c$. Take
\[
 A=N\cup\{\widetilde c:\, c\in C\setminus\{0\}\}.
\]
We claim that 
\begin{equation} \label{eq_order_pi_PP}
\PP(A)=\pi^{-1}(\PP(C)). 
\end{equation}
Indeed, if $x\in \PP(A)$ is expressed as $x= \sum_{b\in N'}b+\sum_{c\in C'}\widetilde{c}$ for some finite sets $N'\subseteq N$ and $C'\subseteq C\setminus \{0\}$, then 
\[
\pi(x) = 0+\sum_{c\in C'}\pi(\widetilde{c})=\sum_{c\in C'}c \in \PP(C).
\]
Conversely, assume \(\pi(x)\in\PP(C)\). Then there is a finite set \(C'\subseteq C\) such
that $\pi(x)=\sum_{c\in C'}c$. The element \(0\) does not affect this sum. So we may assume without loss of generality that \(C'\subseteq C\setminus\{0\}\). Now take $y:=\sum_{c\in C'}\widetilde c$. Then $\pi(y)=\pi(x)$, and hence $n:=x-y\in N$. If $n\neq 0$, then 
\[
x=n+\sum_{c\in C'}\widetilde c \in \PP(A),
\]
noting that the summands in this expression are distinct. If \(n=0\), then simply $x=\sum_{c\in C'}\widetilde c\in\PP(A)$. Now \eqref{eq_order_pi_PP} follows.

Since $\pi(G_i)=Q_i$, the map $\pi_i:=\pi|_{G_i}:\, G_i\rightarrow Q_i$ is surjective and $\ker\pi_i=N\cap G_i$. By \eqref{eq_order_pi_PP},
\[
\PP(A)\cap G_i = \pi^{-1}\bigl(\PP(C)\bigr)\cap G_i = \pi_i^{-1}\bigl(\PP(C)\cap Q_i\bigr).
\]
Moreover, every fibre of \(\pi_i\) is a coset of \(N\cap G_i\). We conclude that
\[
\frac{|\PP(A)\cap G_i|}{|G_i|} = \frac{|\PP(C)\cap Q_i|\,|N\cap G_i|}{|Q_i|\,|N\cap G_i|} =\frac{|\PP(C)\cap Q_i|}{|Q_i|}.
\]
Since the two density sequences agree term by term, their sets of
subsequential limits agree as well.
\end{proof}

Next, we identify four quotient types that suffice for the realization theorem. For a prime \(p\), the Pr\"{u}fer \(p\)-group, written additively, is
\[
 \Z(p^\infty):=\Z[1/p]/\Z
 =\{a/p^r+\Z:a\in\Z,\ r\geq0\}.
\]

\begin{lemma} \label{lem:primary-quotients}
Let $P$ be a countably infinite abelian $p$-group.

(1) If $P/pP$ is infinite, then $P$ has a quotient isomorphic to
$\bigoplus_{n=1}^\infty (\Z/p\Z)$.

(2) If $P/pP$ is finite, then $P$ has a quotient isomorphic to
$\Z(p^\infty)$.
\end{lemma}

\begin{proof}
If $P/pP$ is infinite, then it is a countably infinite-dimensional vector space over $\F_p$. Hence $P/pP\cong \bigoplus_{n=1}^\infty(\Z/p\Z)$ which proves the first assertion.

Suppose now that \(P/pP\) is finite. By \cite[Theorem~32.3, p.~137]{FuchsI}, one can choose a basic subgroup \(B\) of \(P\). By \cite[\S33, (\romannumeral1) on p.~139]{FuchsI}, there is some index set \(J\) and positive integers \(e_j\) such that $B\cong\bigoplus_{j\in J}\Z/p^{e_j}\Z$. Moreover, the inclusion \(B\hookrightarrow P\) induces an isomorphism $B/pB\cong P/pP$ by \cite[\S34(B), p.~144]{FuchsI}. On the other hand,
\[
B/pB \cong  \bigoplus_{j\in J} \left(\bigl(\Z/p^{e_j}\Z\bigr)/p\bigl(\Z/p^{e_j}\Z\bigr)\right) \cong
 \bigoplus_{j\in J}\Z/p\Z.
\]
Since \(P/pP\), and hence \(B/pB\), is finite, the index set \(J\) must
be finite. Therefore \(B\) is finite.

Since \(P\) is infinite and \(B\) is finite, the quotient \(P/B\) is
infinite. The structural property \cite[\S33, (\romannumeral3) on p.~139]{FuchsI} of basic subgroups, together with the classification theorem
\cite[Theorem~23.1, pp.~104--105]{FuchsI}, gives
\[
 P/B\cong
 \bigoplus_{\lambda\in\Lambda}\Z(p^\infty)
\]
for some index set \(\Lambda\). Here only Pr\"{u}fer \(p\)-groups can occur
because \(P/B\) is a \(p\)-group. Since \(P/B\) is infinite,
\(\Lambda\neq\varnothing\). Projection onto one of the summands, composed with the natural quotient map, gives a surjective homomorphism $P\twoheadrightarrow P/B\twoheadrightarrow\Z(p^\infty)$. This proves the second assertion.
\end{proof}

\begin{proposition} \label{prop:quotient-alternative}
Let $G$ be a countable locally finite abelian group such that
$|2G|=\infty$.  Then $G$ has a quotient isomorphic to at least one of
the following groups:
\begin{enumerate}
\item[(\romannumeral1)] $\bigoplus_{n=1}^\infty (\Z/4\Z)$;

\item[(\romannumeral2)] $\bigoplus_{n=1}^\infty (\Z/p\Z)$ for some odd prime $p$;

\item[(\romannumeral3)] $\bigoplus_{n=1}^\infty (\Z/p_n\Z)$, where $p_1,p_2,\ldots$ are pairwise distinct odd primes;

\item[(\romannumeral4)] $\Z(p^\infty)$ for some prime $p$.
\end{enumerate}
\end{proposition}

\begin{proof}
Write
\[
 G=\bigoplus_pG_{(p)},
 \qquad
 G_{(p)}:=\{x\in G:p^rx=0\text{ for some }r\geq1\},
\]
for the primary decomposition of $G$, where the direct sum is over all primes $p$, and each $G_{(p)}$ is a $p$-group. Every quotient of a primary
component is also a quotient of $G$, by composition with the canonical
projection onto that component.

Suppose first that $G_{(p)}$ is infinite for some odd prime $p$.  If
$G_{(p)}/pG_{(p)}$ is infinite, then
Lemma~\ref{lem:primary-quotients} gives a quotient of type
(\romannumeral2).  If $G_{(p)}/pG_{(p)}$ is finite, the same lemma gives a
quotient of type (\romannumeral4).

Now suppose that all odd-primary components $G_{(p)}$ are finite. If $G_{(p)}\neq0$ for infinitely many odd primes $p$, choose pairwise distinct odd primes $p_1,p_2,\ldots$ with $G_{(p_n)}\neq0$.  Since $G_{(p_n)}$ is a nonzero finite $p_n$-group, its quotient $G_{(p_n)}/p_nG_{(p_n)}$ is nonzero and there is a surjection $\varphi_n:G_{(p_n)}\twoheadrightarrow\Z/p_n\Z$. Taking the direct sum of these maps and sending all other primary components to zero gives a surjection $G\twoheadrightarrow\bigoplus_{n\geq1}\Z/p_n\Z$. This is type (\romannumeral3).

It remains to consider the case in which the odd-primary part $\bigoplus_{p\geq 3} G_{(p)}$ is finite. Multiplication by $2$ is an automorphism of this odd-primary part, so $|2G|=\infty$ implies that $Q:=2G_{(2)}$ is infinite. If $Q/2Q$ is finite, Lemma~\ref{lem:primary-quotients} gives a surjection $Q\twoheadrightarrow\Z(2^\infty)$. The group $\Z(2^\infty)$ is divisible, and hence injective in the category of abelian groups by \cite[Theorem~21.1, p.~99]{FuchsI}. The surjection therefore extends from $Q\leq G_{(2)}$ to a homomorphism on $G_{(2)}$.  The extension remains surjective, so type (\romannumeral4) follows.

Finally, suppose that $Q/2Q$ is infinite. Let $E:=G_{(2)}/4G_{(2)}$, which is countable and is annihilated by \(4\). Hence, by
the Pr\"{u}fer--Baer theorem for bounded abelian groups
\cite[Theorem~17.2, p.~88]{FuchsI}, there are at most countable index
sets \(I\) and \(J\) such that
\[
E\cong\bigoplus\limits_{i\in I}(\Z/4\Z)\oplus \bigoplus\limits_{j\in J}(\Z/2\Z).
\]
Note that $2E\cong 2G_{(2)}/4G_{(2)}=Q/2Q$. On the other hand, doubling the Pr\"{u}fer--Baer decomposition gives
\[
2E\cong
 \bigoplus\limits_{i\in I}2(\Z/4\Z)\oplus \bigoplus\limits_{j\in J}2(\Z/2\Z) \cong
 \bigoplus\limits_{i\in I}(\Z/2\Z).
\]
Since $Q/2Q$ is infinite, the index set $I$ is infinite. Projection onto the $\Z/4\Z$ summands indexed by $I$ gives a quotient isomorphic to
$\bigoplus_{n=1}^\infty (\Z/4\Z)$. Composing this projection with $G\to G_{(2)}\to E$ gives type (\romannumeral1). This completes the proof.
\end{proof}

\begin{remark}\label{rem:four-classes}
The four alternatives are retained as distinct algebraic families. They are not asserted to be mutually exclusive as quotients of a given group. Types (\romannumeral3) and (\romannumeral4) are both infinite locally cyclic torsion groups and therefore share the construction in Section~\ref{sec:locally-cyclic}.  Types (\romannumeral1) and (\romannumeral2) are treated by similar coordinate constructions in Section \ref{sec:coordinate-cases}. In particular, we will write $\F_p$ instead of $\Z/p\Z$ to emphasize the field structure. 
\end{remark}

\section{Subset sums in locally cyclic groups}\label{sec:locally-cyclic}

The purpose of this section is to prove the following proposition.

\begin{proposition}\label{prop:cyclic-cases}
Let $G$ be a group of type (\romannumeral3) or (\romannumeral4), with a given filtration $G_\bullet$ as in \eqref{eq:def-filtration}.  For every $0\leq\alpha\leq\beta\leq1$, there exists \(A\subseteq G\) such that $\mathcal L_{G_\bullet}(A)=[\alpha,\beta]$. 
\end{proposition}

A group $G$ of type (\romannumeral3) or (\romannumeral4) is locally cyclic. That is to say, every finitely generated subgroup of $G$ is cyclic. In particular, each filtration subgroup $G_i$ is cyclic. For \(y\in G\), we write \(\langle y\rangle\) for the subgroup generated by \(y\).

Moreover, since $G$ is infinite, no finite subgroup can occur infinitely often in an increasing filtration. Deleting repetitions therefore preserves the lower limit, upper limit, and set of subsequential limits of every finite-level density sequence.  Hence, in this section, we always assume that $G_i<G_{i+1}$ for all $i\geq 1$. 

\subsection{Separated extension sets}

Let $H\leq G$. A finite set $E\subseteq G\setminus H$ is said to be
\emph{$H$-separated} if
\[
\PP(E)\cap H=\{0\}.
\]

The separation condition ensures that later extension sets do not create new subset sums in an earlier filtration subgroup.

\begin{lemma}\label{lem:lc-separated-stage-stability}
Let $i_0<i_1<i_2<\cdots$ be a finite or infinite sequence of filtration levels. Let \(A^{[0]}\subseteq G_{i_0}\). For $s\geq 1$, let $E_s\subseteq G_{i_s}\setminus G_{i_{s-1}}$ be \(G_{i_{s-1}}\)-separated sets, and define $A^{[s]}:=A^{[0]}\cup E_1\cup\cdots\cup E_s$. Then 
\begin{equation}\label{eq:lc-stage-stability}
\PP(A^{[t]})\cap G_i
=
\PP(A^{[s]})\cap G_i,\qquad (0\leq s\leq t,\,\,i\leq i_s).
\end{equation}
If the sequence is infinite and $A:=A^{[0]}\cup\bigcup_{s\geq1}E_s$, then
\begin{equation}\label{eq:lc-final-stability}
\PP(A)\cap G_i
=
\PP(A^{[s]})\cap G_i
\qquad (s\geq 0,\,\, i\leq i_s).
\end{equation}
\end{lemma}

\begin{proof}
For \(t\geq1\), the conditions show that $A^{[t-1]}\cap E_t=\varnothing$. Every element of 
\[
\PP(A^{[t]}) 
=
\PP(A^{[t-1]})+\PP(E_t)
\]
has the form \(a+e\), where \(a\in\PP(A^{[t-1]})\subseteq G_{i_{t-1}}\) and \(e\in\PP(E_t)\).  If \(a+e\in G_{i_{t-1}}\), then
\(e=(a+e)-a\in G_{i_{t-1}}\).  Separation therefore forces \(e=0\), and hence
\[
\PP(A^{[t]})\cap G_{i_{t-1}}
=
\PP(A^{[t-1]})\cap G_{i_{t-1}}.
\]
Iterating this identity from \(t\) down to \(s\), and then intersecting with \(G_i\leq G_{i_s}\), proves
\eqref{eq:lc-stage-stability}.

Every element of \(\PP(A)\) is a sum of finitely many elements of
\(A\), so it belongs to \(\PP(A^{[t]})\) for some finite \(t\).
Equation~\eqref{eq:lc-final-stability} now follows from
\eqref{eq:lc-stage-stability} by taking the union over $t\geq s$.
\end{proof}

\subsection{Interval contraction}

The next lemma produces the decreasing part of the density
construction.  It gives an exact density at the terminal filtration
level and uniform bounds at all intermediate filtration levels.

\begin{lemma}\label{lem:lc-interval-contraction}
Let \(i<j\) and $q:=[G_j:G_i]$. Choose a generator \(x\) of \(G_j\), and suppose that $11\leq m\leq q-1$. Then there is a \(G_i\)-separated set $E\subseteq G_j\setminus G_i$ such that, for any \(C\subseteq G_i\),
\begin{equation}\label{eq:lc-contraction-structure}
\PP(C\cup E)
=
\PP(C)+\{0,x,2x,\ldots,(m-1)x\}.
\end{equation}
Moreover, for every filtration level \(\ell\) with
\(i\leq\ell\leq j\),
\begin{equation}\label{eq:lc-contraction-band}
\frac{|\PP(C)|}{|G_i|}\frac{m}{q}
\leq
\frac{|\PP(C\cup E)\cap G_\ell|}{|G_\ell|}
\leq
\frac{|\PP(C)|}{|G_i|}.
\end{equation}
Furthermore, the density at the terminal level \(\ell=j\) is 
\[
\frac{|\PP(C\cup E)|}{|G_j|}=\frac{|\PP(C)|}{|G_i|}\frac{m}{q}.
\]
\end{lemma}

\begin{proof}
Since $m-1\geq 10$, Lemma~\ref{lem_subsetsum_complete_interval} gives a set $D\subseteq \{1,2,\ldots,m-1\}$ with $\PP(D)=\{0,1,\ldots,m-1\}$. Since \([G_j:G_i]=q\), one has $G_i=\langle qx\rangle$. Define $E:=\{kx:\, k\in D\}$. Then $\PP(E)=\{0,x,2x,\ldots,(m-1)x\}$. Because \(m-1<q\), we have $\PP(E)\cap G_i=\{0\}$. Thus \(E\) is \(G_i\)-separated, and \eqref{eq:lc-contraction-structure} follows.

Fix \(\ell\) with \(i\leq\ell\leq j\), and write $d:=[G_\ell:G_i]$. Then $G_\ell=\left\langle\frac{q}{d}x\right\rangle$. Among the integers \(0,1,\ldots,m-1\), precisely $\left\lceil\frac{md}{q}\right\rceil$ are multiples of \(q/d\). Hence exactly this many of the translates
\[
\PP(C),\quad \PP(C)+x,\quad\ldots,\quad
\PP(C)+(m-1)x
\]
lie in \(G_\ell\). These translates occupy distinct \(G_i\)-cosets and therefore are pairwise disjoint.  It follows that
\[
|\PP(C\cup E)\cap G_\ell|
=
|\PP(C)|\left\lceil\frac{md}{q}\right\rceil.
\]
Combining the facts \(|G_\ell|=d|G_i|\) and $md/q \leq
\left\lceil md/q\right\rceil
\leq d$ proves \eqref{eq:lc-contraction-band}. 

When \(\ell=j\), one has \(d=q\), which gives the final assertion.
\end{proof}

\subsection{Two-translate extension sets}

For a finite set \(I\subseteq\mathbb Z\) and \(x\in G\), write
\[
Ix:=\{ax:a\in I\}.
\]

The following lemma produces an extension set whose subset sums give one or two translates by a fixed element in every new coset.

\begin{lemma}\label{lem:lc-two-translate-extension}
Let \(i<j\), and suppose that $\rho:=[G_j:G_i]\geq 100$. Choose a generator \(y\) of \(G_j\).  Let \(0\neq x\in G_i\). Then there is a \(G_i\)-separated set $D\subseteq G_j\setminus G_i$ and nonempty sets $\Delta_s\subseteq\{0,1\}$ $(0\leq s<\rho)$ such that 
\begin{equation}\label{eq:lc-two-translate-decomposition}
\PP(D) = \bigcup_{s=0}^{\rho-1}\bigl(sy+\Delta_sx\bigr).
\end{equation}
Moreover, $\Delta_s=\{0,1\}$ for at least \(\rho-4\) values of \(s\).
\end{lemma}

\begin{proof}
Apply Lemma~\ref{lem_missing_two} with \(T=\rho-3\).  It gives a finite set \(W\subseteq\mathbb N\) such that 
\[
\PP(W)
=
\{0,1,\ldots,\rho-3\}\setminus\{1,\rho-4\}.
\]
In particular, the sum of all elements of \(W\) is \(\rho-3\).

Since \(G_i=\langle\rho y\rangle\) and \(0\neq x\in G_i\), there is an integer \(t\) with $1\leq t<|G_i|$ such that $x=t\rho y$. Define
\[
D:=\{y,y+x\}\cup\{wy:w\in W\} = \big(\{1,1+\rho t\}\cup W\big)y.
\]
Because \(1\notin\PP(W)\), every \(w\in W\) satisfies \(2\leq w\leq\rho-3\). Thus the coefficients of \(y\) occurring in \(D\) are $1$, $1+t\rho$ and $w\in W$, which are distinct modulo \(|G_j|\). Hence, the displayed elements in $D$ are distinct. Moreover, none of these coefficients is divisible by \(\rho\), so every element of \(D\) lies in \(G_j\setminus G_i\).

Consider a nonempty subset of \(D\).  Let \(\varepsilon_0,\varepsilon_1\in\{0,1\}\) record whether \(y\) and \(y+x\), respectively, are selected, and let \(W'\subseteq W\) record the remaining selected elements.  Its sum is
\[
sy+\varepsilon_1x,
\qquad
s:=\varepsilon_0+\varepsilon_1+\sum_{w\in W'}w.
\]
Because the subset is nonempty and the total of \(W\) is \(\rho-3\), one has $1\leq s\leq\rho-1$. Thus the sum does not belong to \(G_i\). Hence $\PP(D)\cap G_i=\{0\}$, and \(D\) is \(G_i\)-separated.

Since $\PP(W)+\{0,1,2\}=\{0,1,\ldots,\rho-1\}$, every value \(s\in\{0,1,\ldots,\rho-1\}\) occurs.  For a fixed \(s\), the possible \(G_i\)-components are \(0\) and \(x\).  This gives nonempty sets \(\Delta_s\subseteq\{0,1\}\) and proves \eqref{eq:lc-two-translate-decomposition}.

Finally, fix \(c\in\PP(W)\), and choose a subset \(W_c\subseteq W\) whose sum is \(c\).  Selecting \(y\) together with the elements \(wy\) with \(w\in W_c\), gives the sum $(c+1)y$, whereas selecting \(y+x\) with the same elements gives $(c+1)y+x$. Therefore \(\Delta_{c+1}=\{0,1\}\).  The values \(c+1\) are distinct, and $|\PP(W)|=\rho-4$. This proves the final assertion.
\end{proof}

\subsection{Expansion estimates}

Let \(i_0\leq i\), let \(x\) be a generator of \(G_{i_0}\), and let \(\varnothing\neq S\subseteq G_{i_0}\).  We say that a set \(T\subseteq G_i\) has \emph{\((S,x)\)-interval sections over \(G_{i_0}\)} if, for every coset \(z+G_{i_0}\) in \(G_i\), there is a nonempty finite interval \(I_z\subseteq\mathbb Z\) such that
\[
T\cap(z+G_{i_0})=z+S+I_zx.
\]

\begin{lemma}\label{lem:lc-one-step-expansion}
Let \(i_0\leq i<j\), let \(x\neq 0\) be a generator of \(G_{i_0}\), and let \(\varnothing\neq S\subseteq G_{i_0}\).  Suppose that \(T\subseteq G_i\) has \((S,x)\)-interval sections over \(G_{i_0}\). Assume that $\rho:=[G_j:G_i]\geq 100$, and let \(D\subseteq G_j\setminus G_i\) be the two-translate extension set from Lemma~\ref{lem:lc-two-translate-extension}, constructed with the same element \(x\). Let
\[
T':=T+\PP(D),
\qquad
u:=\frac{|T|}{|G_i|}.
\]
Then $T'$ has $(S,x)$-interval sections over $G_{i_0}$. Moreover, for every filtration level $\ell$ with $i\leq\ell\leq j$,
\begin{equation} \label{eq_cyclic_twosidedensity}
u
\leq
\frac{|T'\cap G_\ell|}{|G_\ell|}
\leq
u+\frac{|S|}{|G_{i_0}|}.
\end{equation}
At the terminal level,
\begin{equation} \label{eq_cyclic_quantitive_increasing}
\frac{|T'|}{|G_j|}
\geq
u+
\left(1-\frac{4}{\rho}\right)
\frac{1-u}{|G_{i_0}|-1}.
\end{equation}
\end{lemma}

\begin{proof}
Let $y$ and the sets $\Delta_s\subseteq\{0,1\}$ $(0\leq s<\rho)$ be as in Lemma~\ref{lem:lc-two-translate-extension}. Thus every $\Delta_s$ is nonempty, at least $\rho-4$ of them are equal to $\{0,1\}$, and
\[
\PP(D)\cap(sy+G_i)=sy+\Delta_sx.
\]
Since $T\subseteq G_i$, it follows that \begin{equation} \label{eq_T'_Gi_cosets} T'\cap(sy+G_i) = sy+(T+\Delta_sx).
\end{equation}

Choose a set $Z$ of representatives for the cosets of $G_{i_0}$ in $G_i$. Note that \(T\) has \((S,x)\)-interval sections over \(G_{i_0}\). For every $z\in Z$, there is a nonempty finite interval $I_z\subseteq\mathbb Z$ such that 
\begin{equation} \label{eq_T_Gi0_cosets}
T\cap(z+G_{i_0})=z+S+I_zx. 
\end{equation}
Because $x\in G_{i_0}$, the corresponding section of $T'$ is
\begin{equation} \label{eq_T'_Gi0_cosets}
T'\cap(sy+z+G_{i_0}) = sy+z+S+(I_z+\Delta_s)x.
\end{equation}
If $\Delta_s$ is a singleton, then $I_z+\Delta_s$ is a translate of $I_z$. If $\Delta_s=\{0,1\}$, then $I_z+\Delta_s=I_z\cup(I_z+1)$, which is again an interval. Hence $T'$ has $(S,x)$-interval sections over $G_{i_0}$.

We next establish the density estimate \eqref{eq_cyclic_twosidedensity}. Since $\Delta_s\neq\varnothing$, the set $T+\Delta_sx$ contains a translate of $T$, and therefore $|T+\Delta_sx|\geq |T|$. For the upper bound, fix $z\in Z$. If $\Delta_s$ is a singleton, then $|S+(I_z+\Delta_s)x|=|S+I_zx|$. If $I_z=\{a,a+1,\ldots,b\}$ and $\Delta_s=\{0,1\}$, then
\[
S+(I_z+\Delta_s)x
=
\bigl(S+I_zx\bigr)\cup\bigl(S+(b+1)x\bigr).
\]
Consequently, in either case,
\[
|S+(I_z+\Delta_s)x|
\leq
|S+I_zx|+|S|.
\]
Summing the preceding inequality over $z\in Z$ and using \eqref{eq_T_Gi0_cosets} 
gives
\[
|T+\Delta_sx|
\leq
|T|+[G_i:G_{i_0}]|S|.
\]

Fix $i\leq\ell\leq j$, and set $d:=[G_\ell:G_i]$. The group $G_\ell$ is the disjoint union of $d$ cosets of $G_i$. Applying the preceding estimates to the \(d\) cosets of \(G_i\) contained in \(G_\ell\), and using \eqref{eq_T'_Gi_cosets}, we obtain
\[
d|T|
\leq
|T'\cap G_\ell|
\leq
d\bigl(|T|+[G_i:G_{i_0}]|S|\bigr).
\]
Dividing by $|G_\ell|=d|G_i|$ and using $|G_i|=[G_i:G_{i_0}]|G_{i_0}|$, the inequality \eqref{eq_cyclic_twosidedensity} follows.

It remains to prove \eqref{eq_cyclic_quantitive_increasing}. Let $\omega:=|(T+x)\setminus T|$ and
\[
g:=
\#\{s\in\{0,\ldots,\rho-1\}:\, 
\Delta_s=\{0,1\}\}.
\]
If $\Delta_s$ is a singleton, then $|T+\Delta_sx|=|T|$. If $\Delta_s=\{0,1\}$, then $T+\Delta_sx=T\cup(T+x)$, and hence $|T+\Delta_sx|=|T|+\omega$. 
Since the $\rho$ cosets of $G_i$ in $G_j$ are disjoint, one has $|T'|=\rho|T|+g\omega$. Therefore
\[
\frac{|T'|}{|G_j|}
=
u+\frac{g}{\rho}\frac{\omega}{|G_i|}.
\]

For $z\in Z$, set $A_z:=T\cap(z+G_{i_0})$. Every $A_z$ is nonempty. If $A_z\neq z+G_{i_0}$, then $(A_z+x)\setminus A_z\neq\varnothing$. 
Indeed, otherwise $A_z+x\subseteq A_z$. Since translation preserves cardinality, this would imply $A_z+x=A_z$. Because $x$ generates $G_{i_0}$, the set $A_z$ would then be invariant under every translation by an element of $G_{i_0}$, forcing $A_z=z+G_{i_0}$, a contradiction.

Let $r$ be the number of $G_{i_0}$-cosets  in which the section of $T$ is not full. Since translation by $x$ preserves each such coset, the preceding observation gives $\omega\geq r$. Every nonfull section is nonempty and therefore omits at most $|G_{i_0}|-1$ elements. Hence $|G_i|-|T| \leq r\bigl(|G_{i_0}|-1\bigr)$. Consequently, $\omega\geq \frac{|G_i|-|T|}{|G_{i_0}|-1}$. Finally, $g\geq\rho-4$. Substituting these estimates into the formula for $|T'|/|G_j|$ gives
\[
\frac{|T'|}{|G_j|}
\geq
u+
\left(1-\frac{4}{\rho}\right)
\frac{|G_i|-|T|}
{|G_i|(|G_{i_0}|-1)}
=
u+
\left(1-\frac{4}{\rho}\right)
\frac{1-u}{|G_{i_0}|-1}.
\]
This proves the lemma.
\end{proof}

The next lemma shows that repeated two-translate extensions eventually raise the density above any given value less than \(1\). 

\begin{lemma}\label{lem:lc-expansion-to-one}
Let \(i_0<i_1\), and $q:=[G_{i_1}:G_{i_0}]$. Choose a generator \(x\neq 0\) of \(G_{i_1}\).  Let \(\varnothing\neq S\subseteq G_{i_0}\), and let \(1\leq m<q\). Let 
\[
T_1:=S+\{0,x,\ldots,(m-1)x\},
\qquad
u_1:=\frac{|T_1|}{|G_{i_1}|}
=
\frac{m}{q}\frac{|S|}{|G_{i_0}|}.
\]
For every \(h\) satisfying $u_1<h<1$, there is an integer \(N\geq2\), filtration levels $i_1<i_2<\cdots<i_N$, and \(G_{i_{s-1}}\)-separated two-translate extension sets $D_s\subseteq G_{i_s}\setminus G_{i_{s-1}}$ $(2\leq s\leq N)$ with $[G_{i_s}:G_{i_{s-1}}]\geq100$, such that the sets and densities 
\[
T_s:=T_{s-1}+\PP(D_s)\qquad \text{and}
\qquad
u_s:=\frac{|T_s|}{|G_{i_s}|}
\]
satisfy the following estimates: First, 
\begin{equation}\label{eq:lc-expansion-first-crossing}
h\leq u_N < h+\frac1q.
\end{equation}
Second, for \(2\leq s\leq N\) and every filtration level \(\ell\) with \(i_{s-1}\leq\ell\leq i_s\),
\[
u_1 \leq \frac{|T_s\cap G_\ell|}{|G_\ell|} <h+\frac1q.
\]
\end{lemma}

\begin{proof}
Since $G_{i_0}=\langle qx\rangle$ and $m<q$, the sets $S$, $S+x$, $\ldots$, $S+(m-1)x$ lie in distinct $G_{i_0}$-cosets. Hence $|T_1|=m|S|$, which gives the stated formula for $u_1$.

For $s\geq 2$, choose filtration levels $i_1<i_2<i_3<\cdots$ such that $\rho_s:=[G_{i_s}:G_{i_{s-1}}]\geq100$, and let $D_s$ be the two-translate extension set from Lemma~\ref{lem:lc-two-translate-extension}, constructed with the same element $x$, and define $T_s$ and $u_s$ as in the statement.  The set $T_1$ has $(S,x)$-interval sections over $G_{i_1}$. At step $s$, apply Lemma~\ref{lem:lc-one-step-expansion} with base subgroup $G_{i_1}$, current subgroup $G_{i_{s-1}}$, and terminal subgroup $G_{i_s}$. It follows inductively that every $T_s$ has
$(S,x)$-interval sections over $G_{i_1}$.

The estimate \eqref{eq_cyclic_quantitive_increasing} in Lemma~\ref{lem:lc-one-step-expansion} gives
\[
u_s
\geq
u_{s-1}
+
\left(1-\frac{4}{\rho_s}\right)
\frac{1-u_{s-1}}{|G_{i_1}|-1}.
\]
Since $\rho_s\geq100$, the above inequality implies that $1-u_s \leq (1-c)(1-u_{s-1})$, where $c:=(24/25)(|G_{i_1}|-1)^{-1}$. The condition $1\leq m<q$ implies $|G_{i_1}|\geq2$, so $0<c<1$. Iterating the preceding inequality gives
\[
1-u_s
\leq
(1-c)^{s-1}(1-u_1)
\longrightarrow0.
\]
Thus $u_s\longrightarrow1$. 

Let $N\geq2$ be the least index for which $u_N\geq h$. Then $u_{N-1}<h$. Applying the upper bound in Lemma~\ref{lem:lc-one-step-expansion} at the terminal level gives
\[
h
\leq
u_N
\leq
u_{N-1}+\frac{|S|}{|G_{i_1}|}
<
h+\frac{|S|}{|G_{i_1}|}.
\]
Since $S\subseteq G_{i_0}$ and $|G_{i_1}|=q|G_{i_0}|$, we have $\frac{|S|}{|G_{i_1}|}\leq \frac1q$. This proves \eqref{eq:lc-expansion-first-crossing}.

Finally, let $2\leq s\leq N$ and $i_{s-1}\leq\ell\leq i_s$. The intermediate-level estimate in Lemma~\ref{lem:lc-one-step-expansion} gives
\[
u_{s-1}
\leq
\frac{|T_s\cap G_\ell|}{|G_\ell|}
\leq
u_{s-1}+\frac{|S|}{|G_{i_1}|}.
\]
The sequence $(u_s)$ is nondecreasing, while the minimality of $N$ gives $u_{s-1}<h$. Therefore $u_1\leq u_{s-1}<h$. Substitution into the preceding estimate yields
\[
u_1
\leq
\frac{|T_s\cap G_\ell|}{|G_\ell|}
<
h+\frac{|S|}{|G_{i_1}|}
\leq
h+\frac1q.
\]
This proves the lemma.
\end{proof}

\subsection{Construction of the density sequence}

We construct the set in cycles. Each cycle first lowers the density close to $\alpha$ by an interval contraction and then raises it to a target value close to $\beta$ by finitely many two-translate extensions. The last set and filtration level of one cycle are used to begin the next.

\begin{proof}[Proof of Proposition~\ref{prop:cyclic-cases}]
If $\beta=0$, then $\alpha=0$. The conclusion follows by taking $A=\varnothing$ with $\PP(A)=\{0\}$. If $\alpha=1$, then $\beta=1$. The conclusion follows by taking $A=G$ with $\PP(A)=G$. In the following, we assume that $0<\beta\leq1$ and $0\leq\alpha<1$.

Choose a sequence $(h_k)_{k\geq1}$ such that $\alpha<h_k<1$ and $h_k\longrightarrow\beta$. Choose a filtration level $i_{1,0}$ and set
\[
A_{1,0}:=G_{i_{1,0}},
\qquad
x_0:=1.
\]
We now construct the cycles inductively and select sets $A_{k,r}$. Here $k$ indexes the cycle, and \(r\) the stage within it. 

Suppose $A_{k,0}\subseteq G_{i_{k,0}}$ has been chosen. Let
\[
S_k:=\PP(A_{k,0}),
\qquad
x_{k-1}:=\frac{|S_k|}{|G_{i_{k,0}}|}.
\]
For $k=1$, one has $x_0=1$. For $k\geq2$, the preceding construction will give $x_{k-1}\geq h_{k-1}>\alpha$. Thus one always has $x_{k-1}>\alpha$. 

The indices $[G_j:G_{i_{k,0}}]$ are unbounded as $j\to\infty$. We may therefore choose $i_{k,1}>i_{k,0}$ and an integer $m_k$ such
that, writing $q_k:= [G_{i_{k,1}}:G_{i_{k,0}}]$, one has 
\[
\frac{11}{q_k}
<
\min\left\{\frac1k,h_k-\alpha\right\},
\qquad
11\leq m_k<q_k,\qquad \left|
x_{k-1}\frac{m_k}{q_k}-\alpha
\right|
\leq
\frac{11}{q_k}.
\]
Indeed, when $\alpha>0$, one may take $m_k= \left\lfloor \alpha q_k/x_{k-1}\right\rfloor$ for sufficiently large $q_k$. Since $0<\alpha/x_{k-1}<1$, one then has $11\leq m_k<q_k$, and the approximation error is less than $1/q_k$. When $\alpha=0$, take $m_k=11$.

Let $\xi_k$ be a generator of $G_{i_{k,1}}$. Apply Lemma~\ref{lem:lc-interval-contraction} to $A_{k,0}\subseteq G_{i_{k,0}}$, with terminal subgroup $G_{i_{k,1}}$, generator $\xi_k$, and integer $m_k$. Let $E_{k,1}\subseteq G_{i_{k,1}}\setminus G_{i_{k,0}}$ be the resulting $G_{i_{k,0}}$-separated set, and let 
\[
A_{k,1}:=A_{k,0}\cup E_{k,1}.
\]
Then
\[
\PP(A_{k,1})
=
S_k+\{0,\xi_k,\ldots,(m_k-1)\xi_k\}.
\]
Its density at $G_{i_{k,1}}$ is $z_k := |\PP(A_{k,1})|/|G_{i_{k,1}}| = x_{k-1}m_k/q_k$. Therefore $|z_k-\alpha| \leq 11/q_k$. Moreover, we have $z_k \leq \alpha+11/q_k < h_k$, because \(11/q_k<h_k-\alpha\). The contraction estimate \eqref{eq:lc-contraction-band} also gives $z_k
\leq |\PP(A_{k,1})\cap G_\ell|/|G_\ell|\leq x_{k-1}$ for every $i_{k,0}\leq\ell\leq i_{k,1}$. 

We next apply Lemma~\ref{lem:lc-expansion-to-one} with $i_0=i_{k,0}$, $i_1=i_{k,1}$, $S=S_k$, $x=\xi_k$, $m=m_k$ and $h=h_k$. It provides filtration levels $i_{k,1}<i_{k,2}<\cdots<i_{k,N_k}$ and separated two-translate extension sets $E_{k,r}\subseteq G_{i_{k,r}}\setminus G_{i_{k,r-1}}$ $(2\leq r\leq N_k)$. Define
\[
A_{k,r}:=A_{k,r-1}\cup E_{k,r}
\quad
(2\leq r\leq N_k),\qquad 
u_{k,r}
:=
\frac{|\PP(A_{k,r})|}{|G_{i_{k,r}}|}
\quad
(1\leq r\leq N_k).
\]
Since $A_{k,r-1}\subseteq G_{i_{k,r-1}}$ and $E_{k,r}\subseteq G_{i_{k,r}}\setminus G_{i_{k,r-1}}$, the two sets are disjoint. Hence $\PP(A_{k,r}) = \PP(A_{k,r-1})+\PP(E_{k,r})$. Thus these subset-sum sets agree with the recursively defined sets in Lemma~\ref{lem:lc-expansion-to-one}. In particular, $u_{k,1}=z_k$. 

Put $x_k:=u_{k,N_k}$. The conclusions of Lemma~\ref{lem:lc-expansion-to-one} give $h_k\leq x_k< h_k+1/q_k$, and
\[
z_k
\leq
\frac{|\PP(A_{k,r})\cap G_\ell|}{|G_\ell|}
<
h_k+\frac1{q_k}
\]
whenever $2\leq r\leq N_k$ and $i_{k,r-1}\leq\ell\leq i_{k,r}$. This completes the $k$-th cycle. 

Let $i_{k+1,0}:=i_{k,N_k}$ and $A_{k+1,0}:=A_{k,N_k}$. Since $x_k\geq h_k>\alpha$, the construction may be repeated. Finally, define
\[
A:=\bigcup_{k\geq1}A_{k,N_k}.
\]

All sets added after the initial stage are separated from the preceding selected subgroup. Hence Lemma~\ref{lem:lc-separated-stage-stability} shows that the densities computed during the construction are unchanged by all later stages. They are therefore the corresponding densities of the final set $A$.

Since $11/q_k<1/k$, we have \(q_k>11k\), and therefore \(q_k\to\infty\). Consequently,
\[
z_k\longrightarrow\alpha,
\qquad
x_k\longrightarrow\beta.
\]
Thus $\alpha$ and $\beta$ are subsequential limits of the density sequence. Every filtration level in cycle $k$ has density between $z_k$ and $\max\left\{x_{k-1},h_k+1/q_k\right\}$. The lower and upper bounds tend to $\alpha$ and $\beta$, respectively. It follows that
\[
\liminf_{i\to\infty}d_i^{G_\bullet}(\PP(A))=\alpha,
\qquad
\limsup_{i\to\infty}d_i^{G_\bullet}(\PP(A))=\beta,
\]
and hence $\mathcal L_{G_\bullet}(A)\subseteq[\alpha,\beta]$. 

It remains to prove the reverse inclusion. If \(\alpha=\beta\), the conclusion follows. Otherwise, consider any \(\gamma\in(\alpha,\beta)\). For all sufficiently large $k$, one has $z_k<\gamma<h_k$. Since $u_{k,1}=z_k$ and $u_{k,N_k}=x_k\geq h_k$, there is a least $r_k\in\{2,\ldots,N_k\}$ such that $u_{k,r_k}\geq \gamma$. Then $u_{k,r_k-1}<\gamma$. Applying the upper estimate \eqref{eq_cyclic_twosidedensity} in Lemma~\ref{lem:lc-one-step-expansion} to this extension gives
\[
\gamma
\leq
u_{k,r_k}
<
\gamma+\frac{|S_k|}{|G_{i_{k,1}}|}
=
\gamma+\frac{x_{k-1}}{q_k}
\leq
\gamma+\frac1{q_k}.
\]
Hence $u_{k,r_k}\longrightarrow \gamma$ as $k\rightarrow \infty$. The corresponding filtration levels tend to infinity, so $\gamma\in\mathcal L_{G_\bullet}(A)$. Therefore,
\[
\mathcal L_{G_\bullet}(A)=[\alpha,\beta].
\]
This completes the proof.
\end{proof}

\section{Subset sums in coordinate groups} \label{sec:coordinate-cases}

The purpose of this section is to prove the following proposition.

\begin{proposition}\label{prop:coordinate-cases}
Let $G$ be a group of type \textup{(\romannumeral1)} or
\textup{(\romannumeral2)}, with a given filtration $G_\bullet$ as in \eqref{eq:def-filtration}.  For every $0\leq\alpha\leq\beta\leq1$, there exists \(A\subseteq G\) such that $\mathcal L_{G_\bullet}(A)=[\alpha,\beta]$. 
\end{proposition}

As in the preceding section, we may delete repeated terms and assume that $G_i<G_{i+1}$ for every $i\geq1$. Throughout the section, we write
\[
\mathcal R:=\Z/4\Z\quad\text{for type \textup{(\romannumeral1)}},
\qquad
\mathcal R:=\F_p\quad\text{for type \textup{(\romannumeral2)}},
\]
where \(p\) is a given odd prime, and $G=\bigoplus_{n=1}^{\infty}\mathcal R$. Thus \(G\) is a free \(\mathcal R\)-module. We emphasize that, when
\(\mathcal R=\Z/4\Z\), a subgroup \(G_i\leq G\) need not be a free \(\mathcal R\)-module. The submodule generated by an element $a\in G$ is denoted by $\langle a \rangle=\mathcal Ra$.

To state the two constructions simultaneously, let
\[
q_\mathcal R:=|\mathcal R|,
 \qquad
 \gamma_\mathcal R:=
 \begin{cases}
  p,&\mathcal R=\F_p,\\
  2,&\mathcal R=\Z/4\Z.
 \end{cases}
\]

\subsection{A coordinate representation adapted to the filtration}

The first lemma provides one basis that is adapted to every subgroup in the given filtration. 

\begin{lemma} 
\label{prop:coordinate-normal-form}
Let $G$ be a group of type (\romannumeral1) or (\romannumeral2), with a given filtration $G_\bullet$ as in \eqref{eq:def-filtration}. There exist an \(\mathcal R\)-basis \((\mathfrak e_\lambda)_{\lambda\in\N}\) of \(G\) and functions $c,d:\N\longrightarrow\N$ with $1\leq c\leq d$ such that
\begin{equation}\label{eq:unified-filtration-normal-form}
G_i= \bigoplus_{d(\lambda)\leq i}\langle \mathfrak e_\lambda\rangle  \ \oplus\!  \bigoplus_{c(\lambda)\leq i<d(\lambda)} \langle2\mathfrak e_\lambda\rangle, \qquad(i\geq1).
\end{equation}
\end{lemma}

\begin{proof} [Proof of Lemma \ref{prop:coordinate-normal-form}]
First, we consider the case $\mathcal R=\Z/4\Z$. Let $V:=\{x\in G:\,2x=0\}=2G$. Since \(G\) is a free \(\Z/4\Z\)-module, \(V\) is a countably infinite-dimensional vector space over \(\F_2\).  For every \(i\geq1\), let
\[
C_i:=G_i\cap V, \qquad D_i:=2G_i.
\]
Thus \(C_i\) is the subgroup of elements of \(G_i\) annihilated by \(2\), while \(D_i\) consists of the doubles of elements of \(G_i\). Then $(C_i)_{i\geq1}$ and $(D_i)_{i\geq1}$ are increasing filtrations of \(V\) by finite-dimensional \(\F_2\)-subspaces, with $D_i\subseteq  C_i$. 

We next choose an \(\F_2\)-basis of \(V\) simultaneously adapted to the two filtrations \((C_i)\) and \((D_i)\). Let 
\[
 C_0=D_0=\{0\},
 \qquad
 W_{r,s}:=C_r\cap D_s
 \quad(r,s\geq 0).
\]
For every \(r,s\geq1\), choose an \(\F_2\)-subspace \(E_{r,s}\) of \(W_{r,s}\) such that $W_{r,s} = \bigl(W_{r-1,s}+W_{r,s-1}\bigr)\oplus E_{r,s}$. This is possible because all the spaces involved are finite dimensional.  Since $W_{r-1,s}\cap W_{r,s-1}=W_{r-1,s-1}$, induction on \(r+s\) gives 
\[
W_{r,s} =\bigoplus_{\substack{1\leq a\leq r\\ 1\leq b\leq s}}E_{a,b}.
\]
Because both filtrations exhaust \(V\), we have $C_i=\bigcup_{s\geq1}W_{i,s}$ and $D_i=\bigcup_{r\geq1}W_{r,i}$. Consequently, 
\[
C_i= \bigoplus_{a\leq i,\,b\geq1}E_{a,b},
 \qquad
 D_i=
 \bigoplus_{a\geq1,\,b\leq i}E_{a,b}.
\]

Choose a basis of each nonzero space \(E_{a,b}\), and enumerate the union of these bases as $(\mathfrak u_\lambda)_{\lambda\in\N}$. This is an \(\F_2\)-basis of \(V\) simultaneously adapted to \((C_i)\) and \((D_i)\).  Define
\[
 c(\lambda):=\min\{i:\mathfrak u_\lambda\in C_i\},
 \qquad
 d(\lambda):=\min\{i:\mathfrak u_\lambda\in D_i\}.
\]
It follows that
\begin{equation}\label{eq:two-adapted-flags}
 C_i=
 \operatorname{span}_{\F_2}
 \{\mathfrak u_\lambda:c(\lambda)\leq i\},
 \qquad
 D_i=
 \operatorname{span}_{\F_2}
 \{\mathfrak u_\lambda:d(\lambda)\leq i\}.
\end{equation}
Since \(D_i\subseteq C_i\) for every \(i\), we necessarily have $c(\lambda)\leq d(\lambda)$ $(\lambda\in\N)$. By the definition of \(d(\lambda)\), one has $\mathfrak u_\lambda\in D_{d(\lambda)}=2G_{d(\lambda)}$. We may therefore choose $\mathfrak e_\lambda\in G_{d(\lambda)}$ such that $2\mathfrak e_\lambda=\mathfrak u_\lambda$. We claim that \((\mathfrak e_\lambda)_{\lambda\in\N}\) is an \(\mathcal R\)-basis of \(G\).

First, suppose that $\sum_\lambda a_\lambda \mathfrak e_\lambda=0$ is a finite relation with \(a_\lambda\in \mathcal R\). Multiplication by \(2\) gives $\sum_\lambda \overline a_\lambda \mathfrak u_\lambda=0$ over $\F_2$, where \(\overline a_\lambda\) is the reduction of \(a_\lambda\) modulo \(2\). The linear independence of the \(\mathfrak u_\lambda\) implies that \(\overline{a}_\lambda=0\), and hence that every \(a_\lambda\in\{0,2\}\subseteq \mathcal R\). Write \(a_\lambda=2b_\lambda\), where \(b_\lambda\in\{0,1\} \subseteq  \mathcal R\). The original relation now becomes $\sum_\lambda b_\lambda \mathfrak u_\lambda=0$. A second application of the linear independence of the \(\mathfrak u_\lambda\) shows that every \(b_\lambda=0\).  Thus every \(a_\lambda=0\) in \(\mathcal R\), proving that the \(\mathfrak e_\lambda\) are \(\mathcal R\)-linearly independent.

Conversely, let \(x\in G\).  Since \(2x\in V\), there are finitely many coefficients \(\varepsilon_\lambda\in\F_2\) such that $2x=\sum_\lambda\varepsilon_\lambda \mathfrak u_\lambda$. Hence $2\big(x-\sum_\lambda\varepsilon_\lambda \mathfrak e_\lambda\big)=0$, where each  $\varepsilon_\lambda\in \{0,1\}$ is viewed as an element of $\mathcal R$. Therefore,
\[
x-\sum_\lambda\varepsilon_\lambda \mathfrak e_\lambda\in V
 =\operatorname{span}_{\F_2}\{\mathfrak u_\lambda:\lambda\in\N\}
 =\operatorname{span}_{\F_2}\{2\mathfrak e_\lambda:\lambda\in\N\}.
\]
It follows that \(x\) is an \(\mathcal R\)-linear combination of the \(\mathfrak e_\lambda\).  Thus \((\mathfrak e_\lambda)_{\lambda\in\N}\) is an \(\mathcal R\)-basis of \(G\).

It remains to identify the filtration subgroups in this basis.  Define
\[
 H_i:=
 \bigoplus_{d(\lambda)\leq i}\langle \mathfrak e_\lambda\rangle
 \ \oplus\
 \bigoplus_{c(\lambda)\leq i<d(\lambda)}
 \langle2\mathfrak e_\lambda\rangle.
\]
If \(d(\lambda)\leq i\), then $\mathfrak e_\lambda\in G_{d(\lambda)}\subseteq G_i$. If \(c(\lambda)\leq i<d(\lambda)\), then $2\mathfrak e_\lambda=\mathfrak u_\lambda\in C_i\subseteq G_i$. Therefore \(H_i\subseteq G_i\).

For the reverse inclusion, let \(x\in G_i\).  Since \(2x\in D_i\), it follows from \eqref{eq:two-adapted-flags} that $2x=\sum_{d(\lambda)\leq i}
 \varepsilon_\lambda \mathfrak u_\lambda$ 
for suitable \(\varepsilon_\lambda\in\F_2\). Let $y:= x-\sum_{d(\lambda)\leq i}  \varepsilon_\lambda \mathfrak e_\lambda$, with $\varepsilon_\lambda$ viewed as an element from $\{0,1\}\subseteq \mathcal R$. Every \(\mathfrak e_\lambda\) occurring in this sum belongs to \(G_i\), so \(y\in G_i\).  Moreover, \(2y=0\) implies $y\in G_i\cap V=C_i$. Using \eqref{eq:two-adapted-flags} again, we may write
\[
 y=
 \sum_{c(\lambda)\leq i}
 \delta_\lambda \mathfrak u_\lambda
 =
 \sum_{c(\lambda)\leq i}
 2\delta_\lambda \mathfrak e_\lambda,
\]
where each \(\delta_\lambda\in\mathbb F_2\) is represented by an element of \(\{0,1\}\subseteq \mathcal R\) in the second sum. Consequently,
\[
 x\in
 \bigoplus_{d(\lambda)\leq i}\langle \mathfrak e_\lambda\rangle
 \ \oplus\
 \bigoplus_{c(\lambda)\leq i<d(\lambda)}
 \langle2\mathfrak e_\lambda\rangle
 =H_i.
\]
Thus \(G_i=H_i\), proving \eqref{eq:unified-filtration-normal-form}.

Second, we consider the case $\mathcal R=\F_p$. Starting with a basis of $G_1$, recursively extend a basis of $G_{i-1}$ to a basis of $G_i$. The union of the resulting bases is a basis of $G$. Enumerate it as $(\mathfrak e_\lambda)_{\lambda\in\N}$ and set
\[
c(\lambda)=d(\lambda):=\min\{i:\mathfrak e_\lambda\in G_i\}.
\]
Then \eqref{eq:unified-filtration-normal-form} holds.
\end{proof}

\begin{remark*}
The functions $c$ and $d$ record the first filtration levels at which $2\mathfrak e_\lambda$ and $\mathfrak e_\lambda$, respectively, belong to the filtration. More precisely,
\[
c(\lambda)=\min\{i:2\mathfrak e_\lambda\in G_i\},
\qquad
d(\lambda)=\min\{i:\mathfrak e_\lambda\in G_i\}.
\]
If \(\mathcal R=\F_p\), then $c=d$ and hence the second direct sum in \eqref{eq:unified-filtration-normal-form} is empty. If \(\mathcal R=\Z/4\Z\), a coordinate may first contribute its subgroup \(\langle2\mathfrak e_\lambda\rangle\) and may enter in full at a later filtration level. 
\end{remark*}

We call $\mathfrak e_\lambda$ $(\lambda\in \N)$ \textit{basis elements}, which may be denoted by other letters in later subsections. During the construction, we call a basis element \textit{unused} if it has not occurred in previous stages. 

For a finite set $I\subseteq\N$, we call $U_I:=\bigoplus_{\lambda\in I}\langle \mathfrak e_\lambda\rangle$ a finite coordinate summand. Its \emph{dimension} is $|I|$. And we call $I$ the coordinate support of $U_I$.

\subsection{Co-singleton block}

The next lemma provides the basic density-decreasing construction. For convenience, we rewrite a finite coordinate summand $U=\bigoplus_{\lambda\in I}\langle \mathfrak e_\lambda \rangle$ as $U=\bigoplus_{j=1}^m \langle \mathfrak e_j \rangle$ temporarily.

\begin{lemma}
\label{lem:R-cosingleton}
Let \(m\geq 2\), and let $U=\bigoplus_{j=1}^m\langle \mathfrak e_j \rangle$. There is a finite set \(B\subseteq U\) and an element \(v\in U\) of additive order \(q_\mathcal R\) such that 
\begin{equation}\label{eq:R-cosingleton}
 \PP(B)=U\setminus\{v\}.
\end{equation}
They may be chosen with the following additional properties: If \(\mathcal R=\F_p\), then the \(\mathfrak e_m\)-coordinate of \(v\) is nonzero. If \(\mathcal R=\Z/4\Z\), then $v=2\mathfrak e_1+3\mathfrak e_2+\cdots+3\mathfrak e_m$. 
\end{lemma}

\begin{proof} 
First, consider the case \(\mathcal R=\F_p\). In this situation, the conclusion actually holds for all $m\in \N$. We use induction on the dimension $m$. Corollary \ref{cor_missing_one} gives a set $D$ and an element $a\in \F_p$ such that $\PP(D)=\F_p\setminus\{a\}$. Taking $B=D\mathfrak e_1$ and $v=a\mathfrak e_1$ proves the case $m=1$. 

Now we consider an $m$-dimensional coordinate summand  $U=\bigoplus_{j=1}^{m}\langle \mathfrak e_j \rangle$ $(m\geq 2)$, and use the inductive hypothesis on the $m-1$-dimensional coordinate summand $W:=\bigoplus_{j=2}^m \langle \mathfrak e_j \rangle$: there is a set \(D\subseteq W\) and a vector \(v\in W\), with nonzero \(\mathfrak e_m\)-coordinate, such that $\PP(D)=W\setminus\{v\}$. Define
\[
T:=\{\mathfrak e_1+t\mathfrak e_2:t\in\F_p\}, \qquad B:=D\cup T.
\]
Since $D$ and $T$ are disjoint, we have $\PP(B)=\PP(D)+\PP(T)$. 

For \(1\leq s\leq p-1\), Lemma \ref{lem_basic2} shows that the subset sums of \(T\) obtained from \(s\) elements form precisely $s\mathfrak e_1+\langle \mathfrak e_2 \rangle$. Moreover, the empty subset and the full subset of $T$ both sum to zero. Thus, 
\begin{equation}\label{eq:fp-line-block-sums}
 \PP(T)=\{0\}\cup
 \bigcup_{a\in\F_p^*}
 \bigl(a\mathfrak e_1+\langle \mathfrak e_2 \rangle\bigr).
\end{equation}
It follows that 
\[
\PP(B) =  \big(W\setminus \{v\}\big)\cup \bigcup\limits_{a\in \F_p^*}\left(a\mathfrak e_1+\langle \mathfrak e_2 \rangle+(W\setminus \{v\})\right)= U\setminus \{v\}.
\]
This completes the induction. In particular, the element $v$ has additive order $p$ since it is non-zero. 

Second, consider the case \(\mathcal R=\Z/4\Z\). For \(m\geq1\), define \begin{equation}\label{eq:c4-block-defining-set}
B_m:=\{\mathfrak e_m,-\mathfrak e_m\}
 \cup\bigcup_{j=1}^{m-1}
 \{\mathfrak e_j,\mathfrak e_j+\mathfrak e_{j+1},2\mathfrak e_j+\mathfrak e_{j+1}\}.
\end{equation}
For \(m=1\), one has $\PP(B_1)=  U\setminus\{2\mathfrak e_1\}$.

For $U=\bigoplus_{j=1}^m\langle \mathfrak e_j\rangle$ with $m\geq2$, apply the inductive hypothesis to the free $\mathcal R$-submodule
$W:=\bigoplus_{j=2}^m\langle \mathfrak e_j\rangle$. We obtain a set \(D\subseteq W\) and an element $v'$ satisfying $\PP(D)=W\setminus\{v'\}$, where 
\[
D:=\{\mathfrak e_m,-\mathfrak e_m\} \cup\bigcup_{j=2}^{m-1}
 \{\mathfrak e_j,\mathfrak e_j+\mathfrak e_{j+1},2\mathfrak e_j+\mathfrak e_{j+1}\}, \qquad v'=2\mathfrak e_2+3\mathfrak e_3+\cdots+3\mathfrak e_m.
\]
Here, the expression of $v'$ after \(2\mathfrak e_2\) is absent when \(m=2\). Now we have the disjoint union $B_m =T \cup D$, with $T :=\{\mathfrak e_1,\mathfrak e_1+\mathfrak e_2,2\mathfrak e_1+\mathfrak e_2\}$, which gives $\PP(B_m)=\PP(T)+\PP(D)$. 

In $\PP(T)$, for the first coordinates $0, 1, 2, 3$ of $\mathfrak e_1$, the corresponding $\mathfrak e_2$-components are, respectively,
\[
 \{0,2\mathfrak e_2\},\qquad
 \{0,\mathfrak e_2\},\qquad
 \{\mathfrak e_2\},\qquad
 \{\mathfrak e_2,2\mathfrak e_2\}.
\]
Note that the union of two distinct translates of
$W\setminus \{v'\}$ is $W$. Thus all $\mathfrak e_1$-fibres are contained in $\PP(B_m)$ except the fibre with $2\mathfrak e_1$. In this exceptional case, one has 
\[
 (2\mathfrak e_1+\mathfrak e_2)+(W\setminus\{v'\})
 =(2\mathfrak e_1+W)\setminus\{v\},
\]
where $v=2\mathfrak e_1+3\mathfrak e_2+\ldots+3\mathfrak e_m$. This completes the induction. Moreover, one sees that $v$ has order $4$ when $m\geq 2$. This completes the proof.
\end{proof}

We will call the above construction a co-singleton block. In particular, a co-singleton set in a coordinate summand of dimension
$m$ has relative density $1-q_\mathcal R^{-m}$.

\subsection{Auxiliary coordinates}

Let $B\subseteq U$ and $v\in U$ be as in Lemma~\ref{lem:R-cosingleton}, and let $\mathfrak q=\mathfrak e_\mu$ be an unused basis element whose index \(\mu\) does not belong to the coordinate support of $U$. Define
\[
D(\mathfrak q,v):=
\begin{cases}
\{\mathfrak q,\mathfrak q+v,\mathfrak q+2v\},&\mathcal R=\Z/4\Z,\\[1mm]
\{\mathfrak q+tv:t\in\F_p\},&\mathcal R=\F_p.
\end{cases}
\]

\begin{lemma}\label{prop:R-auxiliary}
Let $U$ be a finite coordinate summand of dimension $m\geq 2$, and let $B\subseteq U$ and $v\in U$ be as in
Lemma~\ref{lem:R-cosingleton}. Let $r\geq 0$, and $\mathfrak q_j=\mathfrak e_{\mu_j}$ $(1\leq j\leq r)$ be distinct unused basis elements whose indices lie outside the coordinate support of $U$. Let
\[
Q:=\bigoplus_{j=1}^r\langle \mathfrak q_j\rangle,
\qquad
D:=\bigcup_{j=1}^rD(\mathfrak q_j,v).
\]
Then $D$ is $U$-separated. 

If $\mathcal R=\F_p$, take $E:=Q$.  If $\mathcal R=\Z/4\Z$, let $J\subseteq\{1,\ldots,r\}$ be arbitrary and let
\[
E:=
\bigoplus_{j\in J}\langle \mathfrak q_j\rangle
\oplus
\bigoplus_{j\notin J}\langle2\mathfrak q_j\rangle.
\]
Then
\begin{equation}\label{eq:R-auxiliary-density}
\frac{
\bigl|\PP(B\cup D)\cap(U\oplus E)\bigr|
}{|U|\,|E|}
=
1-q_\mathcal R^{-m}\gamma_\mathcal R^{-r}.
\end{equation}
\end{lemma}

\begin{proof}
The sets $B,D(\mathfrak q_1,v),\ldots,D(\mathfrak q_r,v)$ are pairwise disjoint, and hence their subset-sum sets add.

Suppose first that $\mathcal R=\F_p$. Repeating the one-coordinate argument in the proof of Lemma~\ref{lem:R-cosingleton}, extending the $m$-dimensional coordinate summand $U$ to the $(m+r)$-dimensional coordinate summand $U\oplus Q$, gives $\PP(B\cup D)=(U\oplus Q)\setminus\{v\}$. The same argument shows that \begin{equation}\label{eq:R-auxiliary-separation}
\PP(D)\cap U=\{0\},
\end{equation}
i.e., $D$ is $U$-separated. The missing proportion is $p^{-(m+r)}=q_\mathcal R^{-m}\gamma_\mathcal R^{-r}$, which implies \eqref{eq:R-auxiliary-density}.

Now suppose that $\mathcal R=\Z/4\Z$.  For one auxiliary coordinate $\mathfrak q_j$, the possible $U$-components in the four $\mathfrak q_j$-sections $0,\mathfrak q_j,2\mathfrak q_j,3\mathfrak q_j$ of $\PP(D(\mathfrak q_j,v))$ are
\begin{equation}\label{eq:c4-auxiliary-sections}
\{0\},
\qquad
\{0,v,2v\},
\qquad
\{v,2v,3v\},
\qquad
\{3v\},
\end{equation}
respectively.  If a subset sum from $D$ belongs to $U$, independence of the $\mathfrak q_j$ forces every $\mathfrak q_j$-coefficient to be zero. The displayed
list then forces the contribution of every auxiliary coordinate to be zero. 
This proves \eqref{eq:R-auxiliary-separation}.

Consider the sections over $E$.  Such a section has a unique possible $U$-translation precisely when its $\mathfrak q_j$-coefficient belongs to
$\{0,3\}$ for every $j\in J$ and is $0$ for every $j\notin J$. There are exactly $2^{|J|}$ such sections, and each omits one point
of $U$.  Every other section contains at least two distinct translates of $U\setminus\{v\}$ and is therefore full.  Since $|U|=4^m$ and $|E|=4^{|J|}2^{r-|J|}$, the missing proportion is
\[
\frac{2^{|J|}}{|U|\,|E|}
=
\frac{2^{|J|}}{4^m4^{|J|}2^{r-|J|}}
=
4^{-m}2^{-r}
=
q_\mathcal R^{-m}\gamma_\mathcal R^{-r}.
\]
This proves \eqref{eq:R-auxiliary-density}.
\end{proof}

\subsection{Placement and the product formula}

The next lemma places the coordinate summands into suitable filtration subgroups.

\begin{lemma}\label{lem:R-coordinate-placement}
Let \(N\geq1\), \(m\geq 2\), and let
\(\Lambda_0\subseteq\N\) be finite. Then there exist distinct indices $\lambda_1,\ldots,\lambda_m\in\N\setminus\Lambda_0$ satisfying $c(\lambda_j)>N$ $(1\leq j\leq m)$, together with a finite set
\[
B\subseteq
U:=\bigoplus_{j=1}^m\langle \mathfrak e_{\lambda_j}\rangle
\]
and an element \(v\in U\) of additive order \(q_\mathcal R\), such that $\PP(B)=U\setminus\{v\}$ and
\begin{equation}\label{eq:R-filtration-compatible}
v\in G_i
\quad\Longleftrightarrow\quad
U\subseteq G_i
\qquad(i\geq1).
\end{equation}
\end{lemma}

\begin{proof}
By \eqref{eq:unified-filtration-normal-form}, every index $\lambda$ satisfying $c(\lambda)\leq N$ contributes a nonzero direct summand, either $\langle \mathfrak e_\lambda\rangle$ or $\langle2\mathfrak e_\lambda\rangle$, to the finite group $G_N$. Hence there are only finitely many such indices. Since \(\Lambda_0\) is finite, we may choose distinct indices $\lambda_1,\ldots,\lambda_m\in\N\setminus\Lambda_0$ such that $c(\lambda_j)>N$ $(1\leq j\leq m)$. Let $U:=\bigoplus_{j=1}^m\langle \mathfrak e_{\lambda_j}\rangle$. 

Suppose first that \(\mathcal R=\F_p\). Relabel the chosen indices so that $d(\lambda_j)\leq d(\lambda_m)$ for $1\leq j\leq m$. By Lemma~\ref{lem:R-cosingleton}, there are a finite set \(B\subseteq U\) and an element \(v\in U\) of additive order \(p\) such that $\PP(B)=U\setminus\{v\}$, and the \(\mathfrak e_{\lambda_m}\)-coordinate of \(v\) is nonzero. If \(v\in G_i\), then \eqref{eq:unified-filtration-normal-form} implies
\(d(\lambda_m)\leq i\). Hence $d(\lambda_j)\leq d(\lambda_m)\leq i$ for $1\leq j\leq m$. So \(U\subseteq G_i\).

Now suppose that \(\mathcal R=\Z/4\Z\). Relabel the chosen indices so that $d(\lambda_1)\leq d(\lambda_j)$ for $2\leq j\leq m$. Lemma~\ref{lem:R-cosingleton} gives a finite set \(B\subseteq U\) satisfying $\PP(B)=U\setminus\{v\}$ and $v=2\mathfrak e_{\lambda_1} +3\mathfrak e_{\lambda_2} +\cdots +3\mathfrak e_{\lambda_m}$. In particular, \(v\) has additive order \(4\). Suppose that \(v\in G_i\). For every \(j\geq2\), the \(\mathfrak e_{\lambda_j}\)-coordinate of \(v\) is odd. By \eqref{eq:unified-filtration-normal-form}, this is possible only if $d(\lambda_j)\leq i$. The choice of \(\lambda_1\) then gives $d(\lambda_1)\leq d(\lambda_j)\leq i$, and therefore \(U\subseteq G_i\).

In both cases, the reverse implication is immediate: if \(U\subseteq G_i\), then \(v\in U\subseteq G_i\). Thus \eqref{eq:R-filtration-compatible} follows. 
\end{proof}

We now describe all choices entering the product formula.  Let $\mathcal V$ be a finite or countable index set for the co-singleton blocks. For each $\nu\in\mathcal V$, choose a finite coordinate summand
\[
U_\nu=
\bigoplus_{\lambda\in I_\nu}\langle \mathfrak e_\lambda\rangle,
\qquad
m_\nu:=|I_\nu|\geq 2,
\]
together with the subset $B_\nu\subseteq U_\nu$ and the missing element $v_\nu\in U_\nu$ of order $q_\mathcal R$ such that
\begin{equation}\label{eq:R-compatible-cosingleton}
\PP(B_\nu)=U_\nu\setminus\{v_\nu\},
\qquad
v_\nu\in G_i
\ \Longleftrightarrow\
U_\nu\subseteq G_i\; (i\in \N).
\end{equation}
For each $\nu$, choose a finite or countable family of further basis elements $\mathfrak q_{\nu,j}=\mathfrak e_{\mu_{\nu,j}}$ $(j\in\mathcal J_\nu)$ such that
\begin{equation}\label{eq:R-auxiliary-after-support}
c(\mu_{\nu,j})
>
\max_{\lambda\in I_\nu}d(\lambda).
\end{equation}
Assume that all coordinate supports used for distinct choices are pairwise disjoint, and put
\[
D_{\nu,j}:=D(\mathfrak q_{\nu,j},v_\nu),
\qquad
H_\nu:=
U_\nu\oplus
\bigoplus_{j\in\mathcal J_\nu}\langle \mathfrak q_{\nu,j}\rangle,\qquad 
A_\nu:=
B_\nu\cup
\bigcup_{j\in\mathcal J_\nu}D_{\nu,j}.
\]

For each \(\nu\) and \(i\geq1\), let \(r_\nu(i)\) denote the number of auxiliary coordinates associated with \(B_\nu\) that contribute to the density at level \(G_i\). More precisely, 
\[
r_\nu(i):=
\begin{cases}
\#\{j\in\mathcal J_\nu:2\mathfrak q_{\nu,j}\in G_i\},
 &\mathcal R=\Z/4\Z,\\[1mm]
\#\{j\in\mathcal J_\nu:\mathfrak q_{\nu,j}\in G_i\},
 &\mathcal R=\F_p.
\end{cases}
\]
This number is finite because $G_i$ is finite.

Let $\Omega$ be the set of indices of all unused basis elements at the end of construction, and denote
\[
F:=\bigoplus_{\lambda\in\Omega}\langle \mathfrak e_\lambda\rangle,
\qquad
A_F:=
\bigcup_{\lambda\in\Omega}
\bigl(\langle \mathfrak e_\lambda\rangle\setminus\{0\}\bigr).
\]
Finally, define
\begin{equation}\label{eq:R-final-set}
A:=
A_F\cup
\bigcup_{\nu\in\mathcal V}A_\nu.
\end{equation}

\begin{lemma}\label{prop:R-factor-placement}
For the set $A$ in \eqref{eq:R-final-set}, we have  
\begin{equation}\label{eq:unified-global-factorization}
d_i^{G_\bullet}(\PP(A))
=\prod_{\substack{\nu\in\mathcal V\\U_\nu\subseteq G_i}}
\left(1-q_\mathcal R^{-m_\nu}\gamma_\mathcal R^{-r_\nu(i)}\right),
\qquad(i\geq1).
\end{equation}
Only finitely many factors occur for each fixed $i$.

Moreover, suppose that a finite part of the construction has been chosen. Adjoining any finite or countable family of further co-singleton blocks and auxiliary coordinates, using only basis elements \(\mathfrak e_\lambda\) with \(c(\lambda)>N\), does not change the densities \(d_i^{G_\bullet}(\PP(A))\) for \(i\leq N\).
\end{lemma}

\begin{proof}
Fix $i\geq1$ and $\nu\in\mathcal V$. We mention that $U_\nu \subseteq G_i$ is equivalent to $d(\lambda)\leq i$ for every $\lambda\in I_\nu$. 

Suppose first that $U_\nu\not\subseteq G_i$. By \eqref{eq:R-compatible-cosingleton}, $v_\nu\notin G_i$, and hence 
\[
\PP(B_\nu)\cap G_i=U_\nu\cap G_i.
\]
Moreover, one has $i<\max_{\lambda\in I_\nu}d(\lambda)$, so \eqref{eq:R-auxiliary-after-support} implies that no associated auxiliary coordinate has a nonzero part in $G_i$. The separation condition \eqref{eq:R-auxiliary-separation} shows that these auxiliary sets create no additional subset sum in $G_i$. Thus this value of $\nu$ contributes the factor $1$.

Now suppose that $U_\nu\subseteq G_i$. Let 
\[
\begin{cases}
K_{\nu,i}:=\{j\in\mathcal J_\nu:2\mathfrak q_{\nu,j}\in G_i\},\,\,J_{\nu,i}:=\{j\in K_{\nu,i}:\mathfrak q_{\nu,j}\in G_i\},\qquad &\mathcal R=\Z/4\Z,\\
K_{\nu,i}:=\{j\in\mathcal J_\nu:\mathfrak q_{\nu,j}\in G_i\},\qquad &\mathcal R=\F_p. 
\end{cases}
\]
In either case, $|K_{\nu,i}|=r_\nu(i)$, and this set is finite. The auxiliary coordinate summand contained in $G_i$ is 
\[
\begin{cases}
\bigoplus_{j\in J_{\nu,i}}\langle \mathfrak q_{\nu,j}\rangle \oplus\bigoplus_{j\in K_{\nu,i}\setminus J_{\nu,i}}
\langle2\mathfrak q_{\nu,j}\rangle,\qquad &\mathcal R=\Z/4\Z,\\
\bigoplus_{j\in K_{\nu,i}}\langle \mathfrak q_{\nu,j}\rangle,\qquad &\mathcal R=\F_p.
\end{cases}
\]
After relabelling the coordinates indexed by $K_{\nu,i}$, we apply Lemma~\ref{prop:R-auxiliary} with
$r=|K_{\nu,i}|$. In particular, take the distinguished subset in that lemma to be $J_{\nu,i}$ when $\mathcal R=\Z/4\Z$. This gives
\[
\frac{|\PP(A_\nu)\cap G_i|}{|H_\nu\cap G_i|}
=
1-q_\mathcal R^{-m_\nu}\gamma_\mathcal R^{-r_\nu(i)}.
\]

Any subset sum uses only finitely many auxiliary coordinates. If such a sum lies in $G_i$, the separation condition 
\eqref{eq:R-auxiliary-separation} shows that the coordinates not in $K_{\nu,i}$ make zero total contribution.

Different values of $\nu$ use disjoint coordinate supports. Also, $\PP(A_F)=F$, because every element of $F$ is the sum of its nonzero coordinate components. The expression \eqref{eq:unified-filtration-normal-form} gives
\[
G_i=(F\cap G_i)\oplus
\bigoplus_{\nu\in\mathcal V}(H_\nu\cap G_i),
\]
where only finitely many summands are nonzero. The local cardinalities therefore multiply, which proves \eqref{eq:unified-global-factorization}. Since the nonzero subgroups $U_\nu\subseteq G_i$ have disjoint supports and $G_i$ is finite, only finitely many factors occur.

Finally, fix $i\leq N$. Every newly added co-singleton support is not contained in $G_i$, and every newly added auxiliary count is zero at level $i$. Thus each new local factor in \eqref{eq:unified-global-factorization} is $1$, while all previously chosen factors are unchanged. This proves the stability assertion. It also applies to a countable family, because only finitely many basis coordinates meet the finite group $G_i$. The proof is now complete.
\end{proof}

\subsection{Controlled density changes}

The next lemma provides the decreasing step in the density construction: by adjoining co-singleton blocks on new coordinates, one can approach any smaller target through arbitrarily small decreases.

In the remainder of this section, every selected or terminal filtration
level \(I\) is taken so that every basis element used up to that point
belongs to \(G_I\). This is always possible because only finitely many
basis elements have been used and \(G_\bullet\) exhausts \(G\).
Consequently, all factors belonging to the preceding finite construction
remain constant at every later filtration level.

\begin{lemma}\label{lem:R-controlled-descent}
Suppose that a finite construction has density $x>0$ at a selected filtration level $I$.  Let $0\leq t<x$ and $\delta>0$.  By adding finitely many new co-singleton blocks after level $I$, one obtains a later level with density $x'$ satisfying $t-\delta<x'\leq t$ when $t>0$, and $0<x'<\delta$ when $t=0$. Throughout this descent, the density is nonincreasing, and every nonzero decrease is smaller than $\delta$.
\end{lemma}

\begin{proof}
Choose $m\geq 2$ so large that $w:=q_\mathcal R^{-m}<\delta$. Using Lemma~\ref{lem:R-coordinate-placement}, add co-singleton blocks of dimension $m$, one after another. After each new coordinate summand has become fully contained in the filtration, choose the next one beyond that level.  After $a$ new summands have entered, the product formula gives the density $x(1-w)^a$. 

If \(t>0\), let \(a\) be the least positive integer for which
\(x(1-w)^a\leq t\). The preceding density \(x(1-w)^{a-1}\) is larger
than \(t\), and the decrease from this immediately preceding density is
\[
 x(1-w)^{a-1}-x(1-w)^a
 =x(1-w)^{a-1}w
 \leq w<\delta.
\]
Thus $t-\delta<x(1-w)^a\leq t$. 
The same calculation shows that every preceding nonzero decrease is
smaller than \(\delta\).

If $t=0$, choose $a$ sufficiently large that $x(1-w)^a<\delta$. The proof is complete.
\end{proof}

The next lemma provides the increasing step. It shows that auxiliary coordinates cross any target below $1$, with the excess above the target controlled by the largest current missing proportion.

\begin{lemma}\label{lem:R-controlled-ascent}
Suppose that a finite construction has reached a selected filtration
level $I$ containing every basis element used so far. Let
$\mathcal V_0$ be the nonempty finite family of co-singleton blocks
introduced up to this stage. For each $\nu\in\mathcal V_0$, let
$r_\nu^{(0)}$ be the number of auxiliary coordinates already attached
to the block indexed by $\nu$, and let
\[
M_0:=
\max_{\nu\in\mathcal V_0}
q_\mathcal R^{-m_\nu}\gamma_\mathcal R^{-r_\nu^{(0)}}.
\]
If the density at $G_I$ is $x$ and $x<h<1$, then finitely many further
auxiliary coordinates produce a later selected level with density
$x'$ satisfying
\[
h\leq x'<h+M_0.
\]
Throughout this extension, the density is nondecreasing and lies
between $x$ and $x'$.
\end{lemma}

\begin{proof}
We construct the auxiliary extension successively. After $n$ new
auxiliary coordinates have been added, let $r_\nu^{(n)}$ denote the
total number of auxiliary coordinates attached to the block indexed
by $\nu$, and define its current missing proportion by $\rho_\nu^{(n)}
:=
q_\mathcal R^{-m_\nu}\gamma_\mathcal R^{-r_\nu^{(n)}}$.  Thus
\[
\rho_\nu^{(0)}
=
q_\mathcal R^{-m_\nu}\gamma_\mathcal R^{-r_\nu^{(0)}},
\qquad
M_0=\max_{\nu\in\mathcal V_0}\rho_\nu^{(0)}.
\]
Since every coordinate used so far is contained in the initial
selected subgroup, the product formula gives
\[
x_0:=x
=
\prod_{\nu\in\mathcal V_0}
\left(1-\rho_\nu^{(0)}\right).
\]

Suppose that $n$ auxiliary coordinates have been added and that the
current selected filtration level is $J_n$. Choose
$\nu_n\in\mathcal V_0$ such that $\rho_{\nu_n}^{(n)} = \max_{\nu\in\mathcal V_0}\rho_\nu^{(n)}$. Choose an unused basis element $\mathfrak q=\mathfrak e_\mu$ with $c(\mu)>J_n$, attach the
auxiliary set $D(\mathfrak q,v_{\nu_n})$, and pass to the selected level
$J_{n+1}:=d(\mu)$. Then
\[
r_{\nu}^{(n+1)}
=
\begin{cases}
r_\nu^{(n)}+1,&\nu=\nu_n,\\
r_\nu^{(n)},&\nu\neq\nu_n,
\end{cases}\qquad 
\rho_{\nu}^{(n+1)}
=
\begin{cases}
\gamma_\mathcal R^{-1}\rho_\nu^{(n)},&\nu=\nu_n,\\
\rho_\nu^{(n)},&\nu\neq\nu_n.
\end{cases}
\]

Let $x_n$ be the density at the $n$-th selected level. By the product formula, we have $x_n = \prod_{\nu\in\mathcal V_0}
\big(1-\rho_\nu^{(n)}\big)$. Consequently,
\[
x_{n+1}-x_n
=
\rho_{\nu_n}^{(n)}
(1-\gamma_\mathcal R^{-1})
\prod_{\substack{\nu\in\mathcal V_0\\\nu\neq\nu_n}}
\left(1-\rho_\nu^{(n)}\right)
<
\rho_{\nu_n}^{(n)}.
\]
In particular, $(x_n)_{n\geq0}$ is increasing.

Let $M_n:=\max_{\nu\in\mathcal V_0}\rho_\nu^{(n)}$. We claim that $M_n\to0$. Fix $\varepsilon>0$. Whenever
$M_n\geq\varepsilon$, the selected value
$\rho_{\nu_n}^{(n)}=M_n$ is at least $\varepsilon$ and is divided by
$\gamma_\mathcal R$. For each fixed $\nu\in\mathcal V_0$, this can happen only
finitely many times before its missing proportion becomes smaller
than $\varepsilon$. Since $\mathcal V_0$ is finite, there are only
finitely many indices $n$ for which $M_n\geq\varepsilon$. Therefore, as $n\rightarrow \infty$, one has  $M_n\longrightarrow0$, and $\rho_\nu^{(n)}\longrightarrow0$ $(\nu\in\mathcal V_0)$. Hence 
\[
x_n
=
\prod_{\nu\in\mathcal V_0}
\left(1-\rho_\nu^{(n)}\right)
\longrightarrow1.
\]

Since $h<1$, there is a least positive integer $N$ such that
$x_N\geq h$. Then $x_{N-1}<h$, and
\[
x_N-x_{N-1}
<
\rho_{\nu_{N-1}}^{(N-1)}
\leq M_{N-1}
\leq M_0.
\]
Therefore, with $x':=x_N$, we deduce that $h\leq x'<h+M_0$.

Finally, consider the filtration levels between $J_n$ and $J_{n+1}$.
If $\mathcal R=\F_p$, then $c(\mu)=d(\mu)$, so the new auxiliary coordinate
contributes when it enters in full. If $\mathcal R=\Z/4\Z$, it contributes when
$2\mathfrak q$ enters at level $c(\mu)$, and its contribution does not change
when $\mathfrak q$ enters in full at level $d(\mu)$. In either case, the product
formula shows that the density is nondecreasing from $x_n$ to
$x_{n+1}$. Hence every density occurring during the extension lies
between $x$ and $x'$. This concludes the proof.
\end{proof}

\subsection{Construction of the density sequence}

We proceed as in the proof of Proposition~\ref{prop:cyclic-cases},
using controlled descent and controlled ascent in place of interval
contractions and two-translate extensions. Each cycle first lowers the
density close to $\alpha$ and then raises it to a target $h_k$, where
$h_k\to\beta$.

\begin{proof}[Proof of Proposition~\ref{prop:coordinate-cases}]
The cases $\beta=0$ and $\alpha=1$ are immediate, as in the proof of
Proposition~\ref{prop:cyclic-cases}. We therefore assume that $0<\beta\leq1$ and $0\leq\alpha<1$.

Choose a sequence $(h_k)_{k\geq1}$ such that $\alpha<h_k<1$ and $h_k\longrightarrow\beta$, and choose $\delta_k>0$ such that $\delta_k<
\min\left\{1/k,h_k-\alpha\right\}$. Begin with the initial density $x_0=1$.

Suppose that cycle $k$ begins with density $x_{k-1}>\alpha$. Apply
Lemma~\ref{lem:R-controlled-descent} with target $\alpha$ and error
$\delta_k$, and denote the resulting density by $z_k$. Then $|z_k-\alpha|<\delta_k$ and $z_k<h_k$. 
Every nonzero decrease during this descent is smaller than $\delta_k$,
and every new co-singleton block has missing proportion smaller than
$\delta_k$.

Let $\mathcal V_k$ be the finite family of all co-singleton blocks
introduced up to this point. If $r_{\nu,k}$ auxiliary coordinates are
currently attached to the block indexed by $\nu$, let
\[
M_k:=
\max_{\nu\in\mathcal V_k}
q_\mathcal R^{-m_\nu}\gamma_\mathcal R^{-r_{\nu,k}}.
\]
Apply Lemma~\ref{lem:R-controlled-ascent}, using the maximal-choice
construction in its proof, with initial density $z_k$ and target
$h_k$. This gives a later selected level with density $x_k$ satisfying $h_k\leq x_k<h_k+M_k$. 
Every density occurring during this ascent lies between $z_k$ and
$x_k$. 

We claim that $M_k\to0$. Fix $\varepsilon>0$. For all sufficiently large $k$, one has $\delta_k<\varepsilon$, so every block introduced thereafter has initial missing proportion smaller than $\varepsilon$ and remains below $\varepsilon$. Only finitely many earlier blocks exist. Whenever the current maximum is at least $\varepsilon$, the maximal-choice rule selects one of these earlier blocks and divides its missing proportion by $\gamma_\mathcal R$. This can happen only finitely many times for each block. Since $z_k<h_k$, at least one auxiliary step is performed in every cycle. Hence all current missing proportions are eventually smaller than $\varepsilon$, proving that $M_k\to0$.

Consequently,
\[
z_k\longrightarrow\alpha,
\qquad
x_k\longrightarrow\beta.
\]
Let $A$ be the set in \eqref{eq:R-final-set} associated with the
resulting countable construction. By the stability assertion of
Lemma~\ref{prop:R-factor-placement}, later stages do not change any
density already obtained. Every density occurring in cycle $k$ lies
between $z_k$ and $\max\{x_{k-1},x_k\}$. 
As in the proof of Proposition~\ref{prop:cyclic-cases}, these bounds give
\[
\liminf_{i\to\infty}d_i^{G_\bullet}(\PP(A))=\alpha,
\qquad
\limsup_{i\to\infty}d_i^{G_\bullet}(\PP(A))=\beta.
\]
It follows that $\mathcal L_{G_\bullet}(A)\subseteq[\alpha,\beta]$. 

If $\alpha=\beta$, the preceding equalities show that the density
sequence converges to $\alpha$. Hence $\mathcal L_{G_\bullet}(A)=\{\alpha\}$, and the proof is complete. 
We may therefore assume that $\alpha<\beta$. Fix
$\gamma\in(\alpha,\beta)$. For all sufficiently large $k$, the
controlled descent in cycle $k$ begins above $\gamma$ and ends below
$\gamma$. The first density in this descent that is at most $\gamma$
belongs to $(\gamma-\delta_k,\gamma]$. 
It therefore tends to $\gamma$. Thus every $\gamma\in(\alpha,\beta)$ is a subsequential limit, and consequently
\[
\mathcal L_{G_\bullet}(A)=[\alpha,\beta],
\]
and the proposition follows.
\end{proof}

\section{Proof of Theorems \ref{thm:main}, \ref{thm:finite-2G-obstruction}, and Corollary \ref{cor:sharp-criterion} }\label{sec:proofs}

\begin{proof}[Proof of Theorem~\ref{thm:main}]
By Proposition~\ref{prop:quotient-alternative}, there is a surjective homomorphism $\pi:G\twoheadrightarrow Q$, where $Q$ is one of the four listed groups. Give $Q$ the image filtration $Q_i:=\pi(G_i)$ $(i\geq1)$. Every $Q_i$ is finite, the sequence $(Q_i)_{i\geq1}$ is increasing, and $\bigcup_{i\geq1}Q_i=Q$. Thus $Q_\bullet$ is a filtration of $Q$ by finite subgroups. Repeated terms, if any, may be deleted when applying the results of the preceding sections. This does not change the set of subsequential limits.

If $Q$ is of type \textup{(\romannumeral1)} or
\textup{(\romannumeral2)}, Proposition~\ref{prop:coordinate-cases}
gives a set $C\subseteq Q$ such that $\mathcal L_{Q_\bullet}(C)=[\alpha,\beta]$. If $Q$ is of type \textup{(\romannumeral3)} or \textup{(\romannumeral4)}, the same conclusion follows from Proposition~\ref{prop:cyclic-cases}.

Apply Lemma~\ref{lem:quotient-transfer} to $\pi$ and $C$. It gives a set $A\subseteq G$ such that, for every $i\geq1$, 
\[
\frac{|\PP(A)\cap G_i|}{|G_i|}
=
\frac{|\PP(C)\cap Q_i|}{|Q_i|}.
\]
The two density sequences therefore agree term by term, and hence
\[
\mathcal L_{G_\bullet}(A)
=\mathcal L_{Q_\bullet}(C)
=[\alpha,\beta].
\]
This proves the theorem.
\end{proof}

\begin{proof}[Proof of Theorem~\ref{thm:finite-2G-obstruction}]
Let $K:=2G$, which is a subgroup of $G$. Since $K$ is finite and $\bigcup_iG_i=G$, there is some positive integer $N$ such that $K\subseteq G_i$ for all $i\geq N$. Let $\pi:\, G\longrightarrow G/K:=\overline G$ be the quotient map. The group $\overline G$ is annihilated by 2 and can therefore be viewed as an an $\F_2$-vector space. 

For a given $A\subseteq G$, define $L_A$ to be the vector subspace generated by $\pi(A)$. We claim that 
\begin{equation}\label{eq:boolean-quotient-image}
\pi\bigl(\PP(A)\bigr)=\PP\bigl(\pi(A)\bigr)=L_A.
\end{equation}
Indeed, let $x=a_1+\cdots+a_s\in\PP(A)$, where \(a_1,\ldots,a_s\) are distinct elements of \(A\). In the sum $\pi(x)=\pi(a_1)+\cdots+\pi(a_s)$, repeated images cancel in pairs because \(\overline G\) is annihilated by \(2\). After these cancellations, \(\pi(x)\) is a sum of distinct elements of \(\pi(A)\). Hence $\pi\bigl(\PP(A)\bigr)\subseteq\PP\bigl(\pi(A)\bigr)$. Conversely, let \(b_1,\ldots,b_t\) be distinct elements of \(\pi(A)\). Choose \(a_j\in A\) such that \(\pi(a_j)=b_j\). Since the \(b_j\) are distinct, the \(a_j\) are distinct, and therefore $a_1+\cdots+a_t\in\PP(A)$. Thus
\[
b_1+\cdots+b_t
=\pi(a_1+\cdots+a_t)
\in\pi\bigl(\PP(A)\bigr).
\]
Finally, finite subset sums in an \(\F_2\)-vector space are precisely finite \(\F_2\)-linear combinations. This proves \eqref{eq:boolean-quotient-image}.

For $i\geq N$, let $\overline G_i:=\pi(G_i)=G_i/K$ and define
\[
c_i:=\dim_{\F_2} \bigl(\overline G_i/(L_A\cap\overline G_i)\bigr), \qquad \delta_i:=\frac{|L_A\cap\overline G_i|}{|\overline G_i|} =2^{-c_i}.
\]
The inclusion
\(\overline G_i\subseteq\overline G\) induces the canonical linear embedding
\[
j_i:
\overline G_i/(L_A\cap\overline G_i)
\hookrightarrow
\overline G/L_A,
\qquad
x+(L_A\cap\overline G_i)\longmapsto x+L_A.
\]
Indeed, its kernel is $\{0\}$, and its image is $(\overline G_i+L_A)/L_A$. Since the subspaces \(\overline G_i\) are increasing and exhaust \(\overline G\), the images of the \(j_i\) are increasing and exhaust \(\overline G/L_A\). Consequently, \((c_i)_{i\geq N}\) is nondecreasing.

Let
\[
m:=\dim_{\F_2}(\overline G/L_A)
\in\{0,1,2,\ldots\}\cup\{\infty\}.
\]
Suppose first that \(m<\infty\). Choose $\bar x_1+L_A,\ldots,\bar x_m+L_A$ that form a basis of \(\overline G/L_A\). Since the \(\overline G_i\) exhaust \(\overline G\), their representatives $\bar x_1, \ldots, \bar x_m$ all belong to some common \(\overline G_I\). Hence the image of \(j_I\) contains a basis of \(\overline G/L_A\), so \(j_I\) is surjective and therefore an isomorphism. The same holds for every \(i\geq I\). Thus $c_i=m$ for all $i\geq I$, and consequently \(\delta_i=2^{-m}\) eventually.

Suppose now that \(m=\infty\). For every positive integer \(r\), choose \(r\) linearly independent elements of \(\overline G/L_A\). Their representatives all belong to some common \(\overline G_{I_r}\), so the image of \(j_{I_r}\) contains \(r\) linearly independent elements. Therefore $c_i\geq r$ for all $i\geq I_r$. Since this holds for every \(r\), we have  $c_i\longrightarrow\infty$ and $\delta_i\longrightarrow0$ as $i\rightarrow \infty$.

Next, we claim that
\begin{equation}\label{eq:boolean-fibre-image}
\pi\bigl(\PP(A)\cap G_i\bigr)=L_A\cap\overline G_i  \qquad(i\geq N).
\end{equation}
One inclusion follows from \eqref{eq:boolean-quotient-image}.  Conversely, let $\bar x\in L_A\cap\overline G_i$. Choose $x\in\PP(A)$ with $\pi(x)=\bar x$, and choose $y\in G_i$ with $\pi(y)=\bar x$. Then $x-y\in K\subseteq G_i$, so $x\in G_i$. This proves \eqref{eq:boolean-fibre-image}.

The restriction $\pi|_{G_i}:G_i\to\overline G_i$ has kernel $K$, so each of its fibres has $|K|$ elements. By \eqref{eq:boolean-fibre-image}, every point of $L_A\cap\overline G_i$ has at least one and at most $|K|$ preimages in $\PP(A)\cap G_i$. Therefore
\[
|L_A\cap\overline G_i| \leq |\PP(A)\cap G_i| \leq |K|\,|L_A\cap\overline G_i|.
\]
Since $|G_i|=|K|\,|\overline G_i|$, it follows that
\begin{equation}\label{eq:boolean-density-bound}
\frac{\delta_i}{|K|} \leq d_i^{G_\bullet}(\PP(A)) \leq\delta_i,\qquad (i\geq N).
\end{equation}

If $m=\infty$, then
\eqref{eq:boolean-density-bound} gives $d_i^{G_\bullet}(\PP(A))\to0$ as $i\rightarrow \infty$, which proves the
case $m=\infty$. If $m<\infty$, then $\delta_i=2^{-m}$ eventually, and the same inequality gives
\[
 \frac{2^{-m}}{|K|}
 \leq\underline d_{G_\bullet}\bigl(\PP(A)\bigr)
 \leq\overline d_{G_\bullet}\bigl(\PP(A)\bigr)
 \leq2^{-m}.
\]

Finally, if $\overline d_{G_\bullet}(\PP(A))>1/2$, the upper bound forces $m=0$.
The lower bound then gives
\[
 \underline d_{G_\bullet}\bigl(\PP(A)\bigr)
 \geq\frac1{|K|}=\frac1{|2G|}.
\]
This completes the proof.
\end{proof}
    
\begin{proof} [Proof of Corollary \ref{cor:sharp-criterion}]
Theorem~\ref{thm:main} gives
(\romannumeral1)$\Rightarrow$(\romannumeral3), and (\romannumeral3)$\Rightarrow$(\romannumeral2) is immediate.  If $2G$ is
finite, Theorem~\ref{thm:finite-2G-obstruction} shows that $(0,3/4)$
cannot occur.  Hence (\romannumeral2) implies (\romannumeral1), proving the equivalence.
\end{proof}

\section{Proof of Theorem~\ref{thm:realization-BB}}\label{sec:positive-density-BB}

By Theorem~\ref{thm:main}, there exists $A\subseteq G$ such that $\mathcal L_{G_\bullet}(A)=[\alpha,\beta]$. Then Theorem~\ref{thm:realization-BB} follows from the following theorem, since the corresponding set will satisfy $ \overline d_{G_\bullet}\bigl(\PP(A)\bigr)=\beta>0$.

\begin{theorem}
\label{thm:positive-density-BB}
Let $G$ be a countably infinite locally finite abelian group, and let $G_\bullet$ be any filtration by finite subgroups.  If $A\subseteq G$ satisfies $\overline d_{G_\bullet}\bigl(\PP(A)\bigr)>0$, then there exists an infinite set $B\subseteq\PP(A)$ such that
\[
 B+B\subseteq\PP(A).
\]
\end{theorem}

We now prove Theorem~\ref{thm:positive-density-BB}.  The main ingredient is a density-decay lemma for the set of finite subset sums of an indexed family, under the assumption that it contains no nonzero $x$ together with $2x$.
Indexed families are needed because distinct elements of $A$ may have the same image after passage to a quotient.

\subsection{A density-decay lemma}

Let $\mathbf c=(c_\xi)_{\xi\in I}$ be a countable indexed family in an abelian group.  Repetitions are initially permitted.  For finite $F\subseteq I$, write
\[
 \Sigma(F):=\sum_{\xi\in F}c_\xi,
 \qquad
 \operatorname{FS}(\mathbf c)
 :=\{\Sigma(F):F\subseteq I\text{ finite}\}.
\]
Thus $\operatorname{FS}(\mathbf c)$ includes the empty sum, in agreement with our convention for $\PP(A)$.

\begin{lemma}
\label{lem:no-doubles-decay}
Let $H$ be a countably infinite locally finite abelian group, let $\mathbf c=(c_\xi)_{\xi\in I}$ be a countable indexed family generating $H$, and let $S:=\operatorname{FS}(\mathbf c)$.  Assume that
\begin{equation}\label{eq:BB-ND}
 x\in S,\quad 2x\in S
 \quad\Longrightarrow\quad
 x=0.
\end{equation}
Then, for every filtration $H_\bullet=(H_i)$ of $H$ by finite subgroups,
\begin{equation}\label{eq:BB-Conclusion}
\lim_{i\to\infty}d_i^{H_\bullet}(S)=0.
\end{equation}
\end{lemma}

\begin{proof}
Indices for which $c_\xi=0$ may be deleted without changing $S$.  After doing so, no two distinct indices have the same value: otherwise that value and its double would both be subset sums, contrary to \eqref{eq:BB-ND}.  

A \emph{circuit} is defined to be an inclusion-minimal nonempty finite set $Z\subseteq I$ such that $\Sigma(Z)=0$.  Distinct circuits are disjoint.  Indeed, if distinct circuits $Z,Z'$ intersect, let $U:=Z\setminus Z'$ and $V:=Z'\setminus Z$. Minimality shows that $U$ and $V$ are nonempty, and $\Sigma(U)=\Sigma(V)=x$.  Since $U\cap V=\varnothing$, the union $U\cup V$ represents $2x$.  Condition~\eqref{eq:BB-ND} gives $x=0$, contradicting the minimality of $Z$ and $Z'$.  

Moreover, every finite zero-sum set is a union of circuits, which follows by removing a minimal nonempty zero-sum subset and arguing inductively. Call an index \emph{residual} if it belongs to no circuit. We call each circuit, and each singleton $\{\xi\}$ with $\xi$ residual, a \emph{component}. Then these components partition $I$.

We use the following rigidity observation repeatedly.  Suppose that finitely many coefficients $\varepsilon_\xi\in\{-1,0,1,2\}$ satisfy 
\begin{equation}\label{eq:BB-coefficient-relation}
 \sum_\xi\varepsilon_\xi c_\xi=0.
\end{equation}
Choose finite sets $F,J\subseteq I$ coordinatewise so that $\varepsilon_\xi=2\mathbf 1_F(\xi)-\mathbf 1_J(\xi)$. 
Equation~\eqref{eq:BB-coefficient-relation} says
$2\Sigma(F)=\Sigma(J)$.  Both sides are subset sums, so
\eqref{eq:BB-ND} implies  $\Sigma(F)=\Sigma(J)=0$. 
Consequently, $F$ and $J$ are unions of circuits.  Thus, in every relation of
the form \eqref{eq:BB-coefficient-relation},
\begin{equation}\label{eq:BB-rigidity}
 \varepsilon_\xi=0\text{ at each residual index, and the coefficients are
 constant on each circuit.}
\end{equation}

We call a finite set $F\subseteq I$ a \emph{support} of $s\in S$ if $\Sigma(F)=s$. Such a support is called \emph{circuit-reduced} if it contains no circuit. It is not hard to see that every $s\in S$ has a circuit-reduced support. Also, such circuit-reduced supports are unique. Indeed, suppose that $F_1$ and $F_2$
are circuit-reduced supports satisfying $\Sigma(F_1)=\Sigma(F_2)$. Then
$$
\sum_{\xi\in I}
\bigl(\mathbf 1_{F_1}(\xi)-\mathbf 1_{F_2}(\xi)\bigr)c_\xi=0.
$$
By \eqref{eq:BB-rigidity}, the coefficient is zero at every residual index and is constant on each circuit. A constant value $1$ on a circuit would force that circuit to be contained in $F_1$, while a constant value $-1$ would force it to be contained in $F_2$. Both are impossible because $F_1$ and $F_2$ are circuit-reduced. Hence every coefficient is zero, and therefore $F_1=F_2$.

Since the components partition $I$, at least one of the following two cases appears: there are either infinitely many residual indices or infinitely many circuits.

\textsc{Case 1.} Assume that infinitely many residual indices occur. For any given $m\geq 1$, we take $m$ distinct residual indices $\rho_1,\ldots,\rho_m$. For $w\in\{0,1,2\}^m$, let $a_w:=\sum_{j=1}^m w_jc_{\rho_j}$, and let $T:=\operatorname{FS}\bigl((c_\xi)_{\xi\notin \{\rho_1,\ldots,\rho_m\}}\bigr).
$ Every element of $T$ has a circuit-reduced support avoiding the selected indices. Since no circuit contains a residual index, adjoining any subset of $\{\rho_1,\ldots,\rho_m\}$ preserves circuit reduction. The uniqueness of circuit-reduced supports therefore gives the disjoint union 
\begin{equation}\label{eq:BB-residual-decomposition}
S=\bigcup_{w\in\{0,1\}^m}(a_w+T).
\end{equation}

We enlarge the binary indexing set to 
$$
W_m:=\{w\in\{0,1,2\}^m:
\text{there are no }p<q\text{ with }w_p=0 \text{ and } w_q=2\}.
$$
Thus $W_m$ indexes a larger family of translates of $T$, containing all the translates in \eqref{eq:BB-residual-decomposition}. If exactly $k$ coordinates of $w$ differ from $1$, their values, read from left to right, consist of a string of $2$'s followed by a string of $0$'s. Hence
\begin{equation}\label{eq:BB-W-count}
|W_m|
=\sum_{k=0}^m(k+1)\binom{m}{k}
=(m+2)2^{m-1}.
\end{equation}

We claim that the family $\{a_w+T:\, w\in W_m\}$ is pairwise disjoint. Indeed, if $w,w'\in W_m$, then $w-w'$ cannot contain both $2$ and $-2$, since this would produce a forbidden $0$ before a $2$ in one of the two words. After interchanging $w$ and $w'$ if necessary, all coordinates of $w-w'$ therefore belong to $\{-1,0,1,2\}$. If $a_w+T$ and $a_{w'}+T$ intersected, choosing tail supports $F$ and $F'$ would give
$$
\sum_{j=1}^m(w_j-w'_j)c_{\rho_j}
+\Sigma(F)-\Sigma(F')=0.
$$

This is a relation of the form \eqref{eq:BB-coefficient-relation}. By \eqref{eq:BB-rigidity}, its coefficient at every selected residual index must vanish. Thus $w=w'$, proving the disjointness.

Since the filtration exhausts $H$, there exists $i_0$ such that $c_{\rho_1},\ldots,c_{\rho_m}\in H_{i_0}$. For $i\geq i_0$, every $a_w$ belongs to $H_i$, and hence $(a_w+T)\cap H_i=a_w+(T\cap H_i)$. Now \eqref{eq:BB-residual-decomposition}, \eqref{eq:BB-W-count} and the disjointness give $|S\cap H_i|=2^m|T\cap H_i|$ and $(m+2)2^{m-1}|T\cap H_i|\leq |H_i|$. Consequently,
\begin{equation}\label{eq:BB-residual-bound}
\frac{|S\cap H_i|}{|H_i|}
\leq\frac{2}{m+2}
\end{equation}
for every $i\geq i_0$. Since $m$ is arbitrary, the formula \eqref{eq:BB-Conclusion} holds.

\textsc{Case 2.} Suppose that there are infinitely many circuits. For any given $m\geq 1$, we take $m$ distinct circuits $Z_1,\ldots,Z_m$. For each $1\leq j\leq m$, choose $\zeta_j\in Z_j$, let $z_j:=c_{\zeta_j}$, and define
$$
Q_j:=\{\Sigma(F):F\subsetneq Z_j\},
\qquad
L_j:=\{\Sigma(F):\zeta_j\in F\subsetneq Z_j\},
\qquad
N_j:=z_j+L_j.
$$

Since the zero entries have been deleted, $|Z_j|\geq2$. The proper
subsets of $Z_j$ are precisely its circuit-reduced subsets, so the
uniqueness of circuit-reduced supports gives
\[
|Q_j|=2^{|Z_j|}-1,\qquad 
|L_j|=|N_j|=2^{|Z_j|-1}-1
\geq\frac13|Q_j|.
\]

Moreover, one has $N_j\cap Q_j=\varnothing$. Indeed, an equality
$$
z_j+\Sigma(F)=\Sigma(J),
\qquad
\zeta_j\in F\subsetneq Z_j,\quad J\subsetneq Z_j,
$$

would give a relation whose coefficient vector on $Z_j$ is $\mathbf 1_{\{\zeta_j\}}+\mathbf 1_F-\mathbf 1_J$. 
By \eqref{eq:BB-rigidity}, this vector is constant on $Z_j$. Its value
at $\zeta_j$ is either $1$ or $2$. The value $2$ is impossible at any
other index of $Z_j$, while the value $1$ would force
$Z_j\setminus\{\zeta_j\}\subseteq F$ and hence $F=Z_j$. Both alternatives
are impossible.

Let $T:=\operatorname{FS}\bigl((c_\xi)_{\xi\notin
Z_1\cup\cdots\cup Z_m}\bigr)$. Splitting the unique circuit-reduced support among the selected circuits
and their complement gives
\begin{equation}\label{eq:BB-circuit-decomposition}
S=Q_1+\cdots+Q_m+T,
\end{equation}
where every element of $S$ has a unique representation by the displayed
factors.

Let $s_0:=0$ and $s_k:=z_1+\cdots+z_k$ for $1\leq k\leq m$. Also define $Y_0:=S$ and 
\begin{align*}
Y_k&:=s_k+Q_1+\cdots+Q_{k-1}+L_k
+Q_{k+1}+\cdots+Q_m+T\\
&=s_{k-1}+Q_1+\cdots+Q_{k-1}+N_k
+Q_{k+1}+\cdots+Q_m+T.
\end{align*}
Since $L_k\subseteq Q_k$, subtracting the fixed shift $s_k$ from the first expression for $Y_k$ leaves a restriction of the unique factorization in \eqref{eq:BB-circuit-decomposition}. Hence the displayed parametrization of $Y_k$ is unique.

We claim that $Y_0,\ldots,Y_m$ are pairwise disjoint. Suppose that $Y_l\cap Y_k\neq\varnothing$ for some $0\leq l<k$. Since $s_k-s_l=z_{l+1}+\cdots+z_k$, subtracting the corresponding factor representations gives a relation with coefficients in $\{-1,0,1,2\}$. On $Z_k$, its contribution has the 
form
$$
z_k+u_k-v_k,
\qquad
u_k\in L_k,\quad v_k\in Q_k.
$$
By \eqref{eq:BB-rigidity}, the coefficients are constant on $Z_k$.
Since $\Sigma(Z_k)=0$, it follows that $z_k+u_k=v_k$. This is impossible
because $z_k+u_k\in N_k$, whereas $N_k\cap Q_k=\varnothing$.

Choose $i_0$ such that $c_\xi\in H_{i_0}$ for every $\xi\in Z_1\cup\cdots\cup Z_m$. For $i\geq i_0$, all contributions from the selected circuits and all
shifts $s_k$ belong to $H_i$. Membership in $H_i$ therefore depends only
on the tail contribution from $T$. The uniqueness of the preceding
factorizations gives
$$
|S\cap H_i|
=\Big(\prod_{j=1}^m|Q_j|\Big)|T\cap H_i|
$$
and
$$
|Y_k\cap H_i|
=|L_k|
 \Bigg(\prod_{\substack{1\leq j\leq m\\j\neq k}}|Q_j|\Bigg)
 |T\cap H_i|
\geq\frac13|S\cap H_i|.
$$
Since $Y_0,\ldots,Y_m$ are pairwise disjoint,

$$
|H_i|
\geq\sum_{k=0}^m|Y_k\cap H_i|
\geq\left(1+\frac m3\right)|S\cap H_i|.
$$
Consequently,
\begin{equation}\label{eq:BB-circuit-bound}
\frac{|S\cap H_i|}{|H_i|}
\leq\frac{3}{m+3}
\end{equation}
for every $i\geq i_0$. Since $m$ is arbitrary, the limit in \eqref{eq:BB-Conclusion} holds in this case as well. This completes the proof.
\end{proof}

\subsection{Representations outside finite sets and quotient reduction}

\begin{proof}[Proof of Theorem~\ref{thm:positive-density-BB}]
Let $S:=\PP(A)$ and $H:=\langle A\rangle$, and let $\eta:G\to G/H$ be the quotient map. For each $i$, set $\Gamma_i:=\eta(G_i)$ and $q_i:=|\Gamma_i|=[G_i:G_i\cap H]$. The finite groups $\Gamma_i$ are increasing, so $(q_i)$ is nondecreasing. Since $S\subseteq H$,
$$
d_i^{G_\bullet}(S)
\leq\frac{|G_i\cap H|}{|G_i|}
=\frac1{q_i}.
$$

If $(q_i)$ were unbounded, then $q_i\to\infty$ and $d_i^{G_\bullet}(S)\to0$, contrary to the hypothesis. Thus $(q_i)$ is bounded and hence eventually constant, say $q$. Together with the fact $\Gamma_i\subseteq\Gamma_{i+1}$, one deduces that $\Gamma_i=\Gamma_{i+1}$ for all sufficiently large $i$. In view of $\bigcup_{i\geq1}\Gamma_i=G/H$, we then get $[G:H]=q<\infty$. In particular, $H$ is infinite. The groups $H_i:=H\cap G_i$ $(i\geq 1)$ form a filtration $H_\bullet$ of $H$. Since $S\subseteq H$, one has $|S\cap H_i|/|H_i| =q_i |S\cap G_i|/|G_i|$. Thus $\overline d_{H_\bullet}(\PP(A))= q  \, \overline d_{G_\bullet}(\PP(A))>0$. Replacing $G$ by $H$ and $G_i$ by $H_i$, we may assume without loss of generality that $\langle A\rangle=G$.

Let $R$ be the set of all $r\in G$ such that, for every finite $D\subseteq A$, there is a finite $F\subseteq A\setminus D$ satisfying $\sum_{a\in F}a=r$. For $r\neq0$, this is equivalent to having infinitely many pairwise disjoint supports representing $r$. 

The set $R$ is a subgroup of $G$.  Indeed, representations of $r,r'\in R$ can be chosen successively with disjoint supports, and their union represents $r+r'$, which implies $r+r'\in R$. Noting that $G$ is locally finite, each $r\in R$ has a finite order, say $n$. Then the union of $n-1$ such representations of $r$ represents $(n-1)r=-r$, which gives $-r\in R$.  

Now we claim that  
\begin{equation}\label{eq:BB-R-stabilizer}
 R\subseteq S,
 \qquad
 S+R=S.
\end{equation}
The first assertion follows by taking $D=\varnothing$.  For the second, represent each $r\in R$ outside a support of each fixed $s\in S$ to obtain $s+r\in S$. Hence $S+R\subseteq S$. The reverse inclusion also holds because $0\in R$.

If $R$ is infinite, then $B:=R$ satisfies $B\subseteq S$ and $B+B=R\subseteq S$.  In the following, we suppose that $R$ is finite. 

Let $\pi:G\longrightarrow\overline G:=G/R$, and $\overline S:=\pi(S)$. Since $R$ is finite, $R\subseteq G_i$ for all sufficiently large $i$. After discarding the preceding levels, let $\overline G_i:=\pi(G_i)=G_i/R$.  Then $(\overline G_i)$ is a filtration of $\overline G$. By \eqref{eq:BB-R-stabilizer}, one has $s+R\subseteq S$ for any $s\in S$. Thus, $S$ is a union of $R$-cosets. Now both $S\cap G_i$ and $G_i$ are unions of $R$-cosets. We conclude that
\begin{equation}\label{eq:BB-quotient-density}
 \frac{|\overline S\cap\overline G_i|}{|\overline G_i|}
 =\frac{|S\cap G_i|}{|G_i|}.
\end{equation}
Hence $\overline S$ has positive upper density.

A \emph{quotient certificate} is defined to be a pair $(F,J)$ of finite subsets of $A$ such that, with $b:=\sum_{a\in F}a$,
\begin{equation}\label{eq:BB-quotient-certificate}
 \pi(b)\neq0,
 \qquad
 \pi\Big(\sum_{a\in J}a\Big)=2\pi(b).
\end{equation}
Its \emph{support} is defined to be $F\cup J$.  Since $A$ is countable, enumerate all quotient certificates and greedily retain one whenever its support is disjoint from those retained
earlier.  The resulting family is maximal among families of certificates with pairwise disjoint supports.

Suppose first that there are infinitely many certificates $(F_n,J_n)$ with pairwise disjoint supports, and let $b_n:=\sum_{a\in F_n}a$. We claim that the elements $\pi(b_n)$ take infinitely many distinct values. Otherwise, since $\pi(b_n)\neq0$ for every $n$, some nonzero element of $G/R$ would equal $\pi(b_n)$ for infinitely many $n$. The corresponding $b_n$ would all lie in a single $R$-coset. Since $R$ is finite, some $b\in G$ would occur as $b_n$ infinitely often. The corresponding sets
$F_n$ are pairwise disjoint representations of $b$, so $b\in R$, contrary to $\pi(b)\neq0$.

After passing to a subsequence, we may assume that the elements $\pi(b_n)$ are pairwise distinct. Let $B:=\{b_n:n\geq1\}$. Then $B$ is an infinite subset of $S$. If $n\neq m$, the disjointness of the certificate supports implies that $F_n\cap F_m=\varnothing$. Hence, $F_n\cup F_m$ is a representation of $b_n+b_m$, and therefore
$b_n+b_m\in S$. Moreover, by \eqref{eq:BB-quotient-certificate},
\[
2b_n-\sum_{a\in J_n}a\in R.
\]
In view of \eqref{eq:BB-R-stabilizer} and the fact $\sum_{a\in J_n}a\in S$, one has $2b_n\in S$. Consequently, $B+B\subseteq S$.

It remains to consider the case in which the maximal family of quotient certificates with pairwise disjoint supports is finite. Let $U\subseteq A$ be the union of its supports. By maximality, $U$ meets the support of every quotient certificate. Let $C:=A\setminus U$ and $T:=\pi\bigl(\PP(C)\bigr)$. Then
\begin{equation}\label{eq:BB-quotient-ND}
 x\in T,\quad 2x\in T
 \quad\Longrightarrow\quad
 x=0.
\end{equation}
Indeed, if $x\neq0$, supports $F,J\subseteq C$ representing $x$ and $2x$ would form a quotient certificate whose support is disjoint from $U$.

Let $\overline H:=\langle\pi(C)\rangle$.  Since $A=C\cup U$ is a disjoint union and generates $G$,
\[
 \overline G=\overline H+\langle\pi(U)\rangle.
\]
The subgroup $\langle\pi(U)\rangle$ is finite because $U$ is finite and $\overline G$ is locally finite.  Hence $\overline H$ has finite index in $\overline G$.  Since $\overline G$ is infinite, so is $\overline H$.

Apply Lemma~\ref{lem:no-doubles-decay} to the indexed family $(\pi(c))_{c\in C}$ in $\overline H$, with the filtration $\overline H_i:=\overline H\cap\overline G_i$. By \eqref{eq:BB-quotient-ND}, one obtains that $|T\cap\overline H_i|/|\overline H_i|\longrightarrow0$. Since $T\subseteq\overline H$,
\[
 0\leq
 \frac{|T\cap\overline G_i|}{|\overline G_i|}
 \leq
 \frac{|T\cap\overline H_i|}{|\overline H_i|}
 \longrightarrow0.
\]

Finally, let $P:=\pi(\PP(U))$.  Since $A=C\cup U$, one has $\overline S=P+T$. The set $P$ is is finite and is therefore contained in $\overline G_i$ for all sufficiently large $i$.  For such $i$ and every $p\in P$, one has $(p+T)\cap\overline G_i=p+(T\cap\overline G_i)$. Consequently, $|\overline S\cap\overline G_i| \leq |P|\,|T\cap\overline G_i|$, and therefore
\[
 \frac{|\overline S\cap\overline G_i|}{|\overline G_i|}
 \longrightarrow0.
\]
This contradicts \eqref{eq:BB-quotient-density} and the positive upper density of $S$. This completes the proof.
\end{proof}

\section{Some examples of locally finite abelian groups}\label{sec:examples-filtrations}

The proofs in Sections~\ref{sec:locally-cyclic} and
\ref{sec:coordinate-cases} already treat direct sums of finite cyclic
groups and Pr\"ufer groups. We now present two concrete examples with
their natural filtrations to illustrate the scope of the main results.
All rings in this section are considered as additive groups. In each
example, Theorem~\ref{thm:main} gives, for every prescribed
filtration $G_\bullet$ and every $0\leq\alpha\leq\beta\leq1$, a set
$A\subseteq G$ such that
\[
 \mathcal L_{G_\bullet}(A)=[\alpha,\beta].
\]
If $\beta>0$, Theorem~\ref{thm:realization-BB} also gives an infinite
set $B\subseteq\PP(A)$ with $B+B\subseteq\PP(A)$. The filtrations below
illustrate these conclusions; the theorems apply to every filtration by
finite subgroups.

For abelian groups, local finiteness is equivalent to every element
having finite order. Thus, $\Z$ is outside this framework. A fixed finite
field or finite ring is locally finite as an additive group, but its
filtrations eventually stabilize, and it cannot satisfy $|2G|=\infty$.

\subsection{Polynomials over \texorpdfstring{$\Z/m\Z$}{Z/mZ}}

Let $m\geq3$ be an integer and let $G=(\Z/m\Z)[t]$. Its degree
filtration is
\[
 G_i=\left\{\sum_{j=0}^{i-1}a_jt^j:a_j\in\Z/m\Z\right\},
 \qquad |G_i|=m^i,
 \qquad [G_{i+1}:G_i]=m.
\]
Every polynomial belongs to some $G_i$, so these finite subgroups
exhaust $G$. In particular, $G$ is countably infinite and locally finite.
If $m$ is odd, multiplication by $2$ is bijective, and hence $2G=G$.
If $m$ is even, then $m\geq4$, and $2G$ is infinite because it contains
the distinct elements $2t^j$ for $j\geq0$. In this case, reduction of
coefficients modulo $2$ gives
\[
 G/2G\cong(\Z/2\Z)[t],
\]
so $[G:2G]=\infty$. Thus, the hypothesis $|2G|=\infty$ holds for every
$m\geq3$, including examples in which doubling is not surjective.

\subsection{Rationals modulo integers}

Let $G=\mathbb Q/\Z$. A natural filtration is
\[
 G_i=(i!)^{-1}\Z/\Z
 =\left\{\frac{a}{i!}+\Z:a\in\Z,\ 0\leq a<i!\right\},
 \qquad |G_i|=i!,
 \qquad [G_{i+1}:G_i]=i+1.
\]
These groups are nested, and every rational number has a denominator
dividing some factorial. Hence, their union is $G$. Each $G_i$ is cyclic,
while the successive indices are unbounded. Moreover, every
$r+\Z\in G$ is twice $r/2+\Z$, so $2G=G$ is infinite.

More generally, every finite subgroup of $\mathbb Q/\Z$ has the form
$n^{-1}\Z/\Z$ for a positive integer $n$. Consequently, its filtrations
are exactly the chains
\[
 G_i=n_i^{-1}\Z/\Z,
 \qquad n_i\mid n_{i+1},
\]
in which every positive integer divides some $n_i$. The successive
indices are $n_{i+1}/n_i$, and the realization conclusion holds for every
such choice.

\section{An open question}\label{sec:open-questions}

Theorem~\ref{thm:positive-density-BB} gives an infinite set $B$ with
$B+B\subseteq\PP(A)$ whenever $\PP(A)$ has positive upper density.
It is natural to ask which additional arithmetic or geometric properties
can be imposed on $B$ while retaining this inclusion. We proppose one concrete
question concerning prescribed quadratic distances.

Let $p\geq5$ be a prime and let
\[
 G=(\Z/p\Z)^\infty:=\bigoplus_{j\geq1}\Z/p\Z.
\]
Its elements are sequences with only finitely many nonzero coordinates.
We use the standard filtration
\[
 G_i=\{(x_1,\ldots,x_i,0,\ldots):x_j\in\Z/p\Z\},
 \qquad |G_i|=p^i.
\]
For $x=(x_j)_{j\geq1}\in G$, define
$\|x\|:=\sum_{j\geq1}x_j^2\in\Z/p\Z$. This sum is finite because $x$
has finite support, and the arithmetic is taken modulo $p$.

\begin{question}\label{q:equilateral-BB}
Suppose that $A\subseteq G$ satisfies
$\overline d_{G_\bullet}\bigl(\PP(A)\bigr)>0$.
For every prescribed $t\in(\Z/p\Z)\setminus\{0\}$, does there exist an infinite
set $B\subseteq G$ such that
\[
 B+B\subseteq\PP(A),
 \qquad
 \|b-b'\|=t\quad\text{for all distinct }b,b'\in B?
\]
\end{question}

We recall the finite-dimensional background of the Erd\H{o}s--Falconer distance
problem. For $E\subseteq(\Z/p\Z)^d$, $d\geq2$, write
\[
 \Delta(E):=\left\{\sum_{j=1}^d(x_j-y_j)^2:x,y\in E\right\}.
\]
Iosevich and Rudnev~\cite{IR07} developed a Fourier-analytic approach
to this problem and proved that
\[
 |E|\geq Cp^{(d+1)/2}
 \quad\Longrightarrow\quad
 \Delta(E)=\Z/p\Z,
\]
where $C$ is a sufficiently large absolute constant. Subsequent work
sought smaller size thresholds for determining a positive proportion of
all distances. In the plane over an odd prime field, Murphy, Petridis,
Pham, Rudnev and Stevens~\cite{MPPRS22} obtained such a conclusion
under the hypothesis $|E|\geq p^{5/4}$; their result also gives a positive
proportion of distances from a single point of $E$.

More recently, Pham, Shen, Tran and Xue~\cite{PSTX26} reached the
conjectured threshold exponent $d/2$, up to logarithmic losses, in every
even dimension over odd prime fields. In particular, they proved that
\[
 |E|>p^{d/2}
 \quad\Longrightarrow\quad
 |\Delta(E)|\geq \frac{cp}{\log^2(2p)},
 \qquad d\geq2\text{ even},
\]
with an absolute constant $c>0$. Their theorem applies more generally
to every nondegenerate quadratic form. 

Question~\ref{q:equilateral-BB} asks for one infinite equilateral set
with a prescribed nonzero squared distance, subject to the additional
constraint on its sums. The inclusion $B+B\subseteq\PP(A)$ includes
every double $2b$. Theorem~\ref{thm:positive-density-BB} and the infinite
Ramsey theorem already give a common value of $\|b-b'\|$, but that value
may be zero and is not prescribed.

It would also be interesting to consider other well-known questions in the setting of locally finite abelian groups, for example, problems on distances, incidences, and sum-product phenomena. On the sum-product type question, if $G=(\mathbb{Z}/p\mathbb{Z})[t]$, it follows from \cite{bl} that uniform power growth exists in this group. For a general locally finite abelian group, one first needs to determine a suitable multiplication, as trivial ring structures (such as the zero multiplication, where the product set is always $\{0\}$) do not imply any growth.

\section{Acknowledgements}
Norbert Hegyv\'{a}ri was supported by the National Research, Development and Innovation Office NKFIH Grant No K-146387.

\end{document}